\documentclass[11pt, twoside, leqno]{amsart}  
\usepackage{lipsum}
\usepackage{amsfonts}
\usepackage{graphicx}
\usepackage{epstopdf}
\usepackage{algorithmic}
\usepackage{calligra}
\usepackage[dvipsnames]{xcolor}
\usepackage{amsfonts,amsmath,amsthm,amssymb}
\usepackage{mathtools}
\usepackage{hyperref}
\usepackage[makeroom]{cancel}
\usepackage{autonum}
\usepackage{hhline}
\usepackage{array}
\usepackage{diagbox}
\usepackage{mdframed}
\usepackage{multicol}
\usepackage{graphicx}
\usepackage{subcaption}
\usepackage{moreverb}
\usepackage{bbm}
\usepackage[margin=1.38in]{geometry}
\usepackage{todonotes}
\usepackage{scalerel,amssymb}
\allowdisplaybreaks
\usepackage{mathrsfs}  
\usepackage{lineno}
\usepackage{todonotes}
\usepackage{tikz}
\usepackage{appendix}
\usepackage{enumitem}
\usepackage{pgfplots}
\usetikzlibrary{arrows.meta}
\usepackage[numbers,sort&compress]{natbib}
\definecolor{mygreen}{HTML}{43a047}
\usepackage{subcaption}
\usepackage{doi}

\newcommand{\Om}{\Omega}

\newcommand{\rhob}{\rho_{\textup{b}}}
\newcommand{\rhoa}{\rho_{\textup{a}}}
\newcommand{\Ca}{C_{\textup{a}}}
\newcommand{\Cb}{C_{\textup{b}}}
\newcommand{\kappaa}{\kappa_{\textup{a}}}
\newcommand{\Thetaa}{\Theta_{\textup{a}}}

\newcommand{\ddt}{\frac{\textup{d}}{\textup{d}t}}

\newcommand{\ds}{\, \textup{d} s }
\newcommand{\dx}{\, \textup{d} x}

\newcommand{\intO}{\int_{\Omega}}

\newcommand{\R}{\mathbb{R}} 
 
\newtheorem{theorem}{Theorem}
\newtheorem{lemma}{Lemma}
\newtheorem{proposition}{Proposition}
\newtheorem{assumption}{Assumption}

\numberwithin{lemma}{section}
\numberwithin{proposition}{section}
\numberwithin{theorem}{section}
\numberwithin{equation}{section}
\makeatletter
\newcommand{\leqnomode}{\tagsleft@true}
\newcommand{\reqnomode}{\tagsleft@false}
\makeatother

\definecolor{grey}{rgb}{0.5,0.5,0.5}
   
\title[Westervelt--hyperbolic Pennes system]{Global existence and asymptotic behavior of the Westervelt--hyperbolic Pennes system}      
     
             
\author[I. Benabbas]{Imen Benabbas$^\dagger$}
\thanks{$^\dagger$AMNEDP Laboratory, Faculty of Mathematics,
	USTHB (\href{ibenabbas@usthb.dz}{ibenabbas@usthb.dz})}
\author[B. Said-Houari]{Belkacem Said-Houari$^\ddag$}
\thanks{$^\ddag$Department of Mathematics, College of Sciences, University of
	Sharjah, P. O. Box: 27272, Sharjah, United Arab Emirates    (\href{bhouari@sharjah.ac.ae}{bhouari@sharjah.ac.ae})}
\begin{document}
	\vspace*{8mm}
	\begin{abstract}    
In this work, we investigate the global existence and asymptotic behavior of a mathematical model of nonlinear ultrasonic heating based on a coupled system of the Westervelt equation and the hyperbolic Pennes bioheat equation (Westervelt--Pennes--Cattaneo model). First, we prove that the solution exists globally in time, provided that    the lower-order Sobolev norms of the initial data are considered to be small, while the higher-order norms can be arbitrarily large. This is done using a continuity argument together with some interpolation inequalities. Second, we prove an exponential decay of the solution under the same smallness assumptions on the initial data. 
					\end{abstract}   
	\vspace*{-7mm}  
	\maketitle                  
     
\section{Introduction} 

\subsection{The model}
This paper is concerned with the global existence and asymptotic behavior of solutions to the Westervelt--Pennes--Cattaneo system subject to Dirichlet boundary conditions for both temperature and pressure. More precisely, we consider the system   
\begin{equation} 
\left\{ \label{coupled_problem_eq_Cattaneo}
\begin{aligned}
&p_{tt}-c^2(\bar{\Theta})\Delta p - b \Delta p_t = K(\bar{\Theta})\left(p^2\right)_{tt}, \qquad &\text{in} \ \Omega \times (0,T),\\
&\rhoa \Ca\bar{\Theta}_t +\nabla\cdot q + \rhob \Cb W(\bar{\Theta}-\Thetaa) = \mathcal{Q}(p_t), \qquad &\text{in} \ \Omega \times (0,T),
\\  
&\tau q_t+q+\kappaa \nabla \bar{\Theta}=0, \qquad &\text{in} \ \Omega \times (0,T).
\end{aligned}
\right.
\end{equation}
Here $p$ and $\bar{\Theta}$ denotes respectively, the acoustic pressure and the temperature fluctuations. The thermal parameters $\rhoa, \Ca$ and $\kappaa$ are, respectively,  the ambient density, the ambient heat capacity and thermal conductivity of the tissue. The additional term $\rhob \Cb W(\bar{\Theta}-\Thetaa)$ represents the heat loss caused by blood circulation, with  $\rhob, \Cb$ being the density and specific heat capacity of the blood, and $W$ expressing the tissue's volumetric perfusion rate measured in milliliters of blood per milliliter of tissue per second. We denote by $\Thetaa$ the ambient temperature, that is typically taken in the human body to be $37^\circ C$; see \cite{connor2002bio}. The diffusivity of sound is $b$ and $c$ is the speed of sound, which we assume to depend on the temperature $\bar{\Theta}$.  The function $K$ is given by $K(\bar{\Theta})=\beta/(\rho c^2(\bar{\Theta }))$, 
where $\rho$ is the mass density and $\beta $ is the parameter of nonlinearity.

We complement \eqref{coupled_problem_eq_Cattaneo} with the initial conditions 
\begin{equation}\label{Initial_Condi}
p|_{t=0}=p_0,\quad p_t|_{t=0}=p_1,\quad \bar{\Theta}|_{t=0}=\bar{\Theta}_0,\quad q|_{t=0}=q_0
\end{equation}
and Dirichlet-Dirichlet boundary conditions  
\begin{equation} \label{coupled_problem_BC}
p\vert_{\partial \Om}=0, \qquad \bar{\Theta}\vert_{\partial \Om}= \Thetaa.  
\end{equation}

When $c$ and $K$ are constant, the first equation in \eqref{coupled_problem_eq_Cattaneo} reduces to the well-known Westervelt equation, which is a classical model in nonlinear acoustics and was originally derived in \cite{westervelt1963parametric}. It is frequently used for describing the propagation of high-intensity focused ultrasound (HIFU) in thermo-viscous media \cite{Taraldsen, connor2002bio, Doinikov2014SimulationsAM}.
 

The second and third equations in \eqref{coupled_problem_eq_Cattaneo} describe   the evolution of the temperature $\bar{\Theta}$   where the heating source is the acoustic energy absorbed by the tissue represented here by the function $\mathcal{Q}$, which we assume,  to have the form 
\begin{equation} \label{Q_definition}
\mathcal{Q}(p_t)=\frac{2b}{\rhoa c_{a}^4} (p_t)^2
\end{equation}
where  $c_{a}$ is the ambient speed \cite{connor2002bio}. Indeed, together these two equations constitute the hyperbolic version of the Pennes' equation \cite{pennes1948analysis}
(see also \cite[Eq.\ 3]{kabiri2021analysis} and \cite[Eq. 7]{xu2008non})
\begin{equation} \label{Hyperbolic_Pennes}
\begin{aligned}
&\tau \rhoa \Ca \bar{\Theta}_{tt}+(\rhoa \Ca +\tau \rhob \Cb W) \bar{\Theta}_t +\rhob \Cb W (\bar{\Theta}-\Thetaa)-\kappaa \Delta \bar{\Theta} \\
=&\,\mathcal{Q}(p_t)+ \tau \partial_t\mathcal{Q}(p_t),
\end{aligned}
\end{equation}
which we employ here to model heat transfer in biological systems. Its hyperbolic nature comes from adopting the Cattaneo law of heat conduction \cite{Ca48}:  
 \begin{equation}\label{Cattaneo}
\tau q_t+q+\kappaa \nabla \bar{\Theta}=0
\end{equation}
  which differs from the classical Fourier law: 
 \begin{equation}\label{Fourier_Law}
q+\kappaa \nabla \bar{\Theta}=0 
\end{equation}
by the presence of the relaxation term $\tau q_t$, where $\tau $ is the relaxation time parameter. The Cattaneo law was introduced to overcome the paradox of the infinite speed of propagation of thermal signals in the Fourier law; namely equation \eqref{Fourier_Law} implies that any temperature disturbance causes an instantaneous perturbation in the temperature at each point in the medium, which does not accurately represent the physical reality, for instance in situations that involve short heating periods; see \cite{xu2008non},  \cite{kabiri2021analysis}. 

 An important feature of the current model \eqref{coupled_problem_eq_Cattaneo} is the dependence of the speed of sound $c$ on the temperature $\bar{\Theta}$ \cite{connor2002bio, hallaj2001simulations}. 
This connection holds practical significance in various applications. Experimental studies on HIFU therapy have indicated 
  that the temperature elevation in the tissue, induced by the ultrasonic energy, will lead to changes in the physical properties of the tissue, which in turn affect the acoustic and the temperature fields \cite{Choi2011ChangesIU}.

  \subsection{Related literature} 
The Westervelt equation has received considerable attention in recent years and  significant progress has been made recently toward the understanding of its  solutions  and their behavior; see \cite{nikolic2015local, Clason2013AvoidingDI, Kaltenbacher2015MathematicsON, meyer2011optimal, kaltenbacher2019Well, Simonett2016WellposednessAL,kaltenbacher2009global} and the references therein. 
In \cite{kaltenbacher2009global},  the authors proved, under homogeneous Dirichlet boundary conditions, that a unique  solution exists globally in time and decays exponentially fast. This was achieved by exploiting the strong damping expressed by the term $-b \Delta p_t, \, b>0$.  In addition,  the initial data were assumed to be small, a constraint that is typically required for  the solutions to nonlinear acoustics equations, even for local well-posedness, to avoid potential degeneracy. Similar well-posedness results for the Westervelt equation subject to other types of boundary conditions, such as Neumann and absorbing boundary conditions, were provided in the papers  \cite{KALTENBACHER20191595, BarbaraKaltenbacher2011, Simonett2016WellposednessAL, nikolic2015local}. We mention also the work \cite{KaltenbacherParabolicApproximation} where the authors addressed the behavior of the solutions of the Westervelt equation when the  sound diffusivity $b\rightarrow 0^+$. They proved that the solution converges at a linear rate, as $b$ goes to zero, to the solution of the wave equation, corresponding to $b=0$.

 For $\tau=0$, system \eqref{coupled_problem_eq_Cattaneo} reduces to the Westervelt--Pennes--Fourier system:
 \begin{equation}\label{Westervelt--Pennes--Fourier}
  \left\{
  \begin{aligned}
	&p_{tt}-c^2(\bar{\Theta})\Delta p - b \Delta p_t = K(\bar{\Theta}) \left(p^2\right)_{tt},\\  
	&\rhoa \Ca\bar{\Theta}_t -\kappaa\Delta \bar{\Theta}+ \rhob \Cb W(\bar{\Theta}-\Thetaa) = \mathcal{Q}(p_t), 
\end{aligned}  
\right.	
\end{equation}
where the second equation is the 
 the parabolic Pennes equation, that is widely used for studying heat transfer in biological systems as it accounts  for both conduction-based heat transfer in tissues and convective heat transfer due to blood perfusion     ~\cite{pennes1948analysis}. 
 

To the best of the authors' knowledge, it seems that the first mathematical study of \eqref{Westervelt--Pennes--Fourier} is the one presented in \cite{Nikolic_2022} where Nikoli\'c and Said-Houari considered \eqref{Westervelt--Pennes--Fourier} with  Dirichlet--Dirichlet boundary conditions, and proved local well-posedness using the energy method together with a fixed point argument.  The work in \cite{Nikolic_2022} was followed by   \cite{NIKOLIC2022628}, where under a smallness assumption on the initial data, the authors established the existence of a global-in-time solution of \eqref{Westervelt--Pennes--Fourier} and proved an exponential decay of the solution. Using the maximal regularity estimate for parabolic systems, Wilke in \cite{Wilke_2023} improved  slightly the regularity assumptions in \cite{Nikolic_2022}.
  
 Recently in \cite{Benabbas_Said_Houar_2023} we investigated \eqref{coupled_problem_eq_Cattaneo} and by employing the energy method together with a fixed point argument, we established the local well-posedness of \eqref{coupled_problem_eq_Cattaneo}. In addition, we showed that \eqref{coupled_problem_eq_Cattaneo} does not degenerate under a smallness assumption on the pressure data in the Westervelt equation.  Furthermore,  we performed a singular limit analysis and proved that the Westervelt--Pennes--Fourier model \eqref{Westervelt--Pennes--Fourier} can 
be seen as an approximation of the Westervelt--Pennes--Cattaneo model \eqref{coupled_problem_eq_Cattaneo}, when the relaxation parameter $\tau$ tends to zero.

\subsection{Main contributions}
The main goal of this work is to study the global existence and asymptotic behavior of the solution to the Westervelt--Pennes--Cattaneo model \eqref{coupled_problem_eq_Cattaneo}. First, using the energy method together with the bootstrap argument,  we show that the solution is global in time, provided that the initial data are small enough. We emphasize here that we assume smallness only on the lower-order  Sobolev norms of the initial data, while the higher-order norms can be arbitrarily large. Second, under the same smallness assumption on the initial data, and by using a Gronwall's type inequality,  we prove that the solution decays exponentially fast to the steady state.       
  \subsection{Outline of the presentation} The rest of this paper is structured as follows: In  Section~\ref{sect2}, we reformulate our problem into a new system that is more convenient for the energy analysis.   We also introduce the necessary theoretical preliminaries and the main assumptions used in the proofs. In  Section \ref{Section_Main_Result}, we state   the main results and give some comments about them.      
 Section~\ref{Section_3} is dedicated to the energy analysis. Achieving the control of the different energies involves delicate estimates, particularly as we aim to assume smallness solely on the lower-order energy. To this end, we apply some interpolation inequalities to facilitate the extraction of these lower-order norms in the nonlinear estimates. In Section \ref{Sec_Global Existence}, we carry out the proofs of the global wellposedness and decay rate of solutions to \eqref{coupled_problem_eq_Cattaneo} by putting together the energy estimates established in Section ~\ref{Section_3}.

 \section{Statement of the problem and main results} \label{sect2}     
To state and prove our result and to lighten the notation, we put 
\begin{equation}
m=\rhoa \Ca\qquad \text{and}\qquad \ell=\rhob \Cb W. 
\end{equation} 
We combine the second and third equations in \eqref{coupled_problem_eq_Cattaneo} to get the equation \eqref{Hyperbolic_Pennes}, then we make the change of variables 
$\Theta=\bar{\Theta}-\Theta_a$
in the temperature so that now the bioheat equation becomes
\begin{equation} \label{bioheat_eq}
\tau m\Theta_{tt} +(m +\tau \ell) \Theta_t + \ell \Theta -\kappaa \Delta \Theta = \mathcal{Q}(p_t) + \tau \partial_t \mathcal{Q}(p_t), \qquad  \text{in} \ \Omega \times (0,T),
\end{equation}
and denoting by   
 \begin{equation}\label{funct_k}
k (\Theta)=K(\Theta+\Thetaa)=\frac{\beta}{\rho c^2(\Theta+\Thetaa)} \qquad \text{and}\qquad  h (\Theta)=c^2(\Theta+\Thetaa),  
\end{equation}
we recast the pressure equation in \eqref{coupled_problem_eq_Cattaneo} as 
\begin{equation} \label{pressure_eq}
(1- 2k (\Theta) p)p_{tt}-h (0)\Delta p - b \Delta p_t = 2k (\Theta) (p_{t})^2+\tilde{h} (\Theta)\Delta p, \qquad \text{in} \ \Omega \times (0,T)
\end{equation}
where we have decomposed $h (\Theta)$ as follows
\begin{equation}
h (\Theta)=h(0)+ \tilde{h} (\Theta)  \quad   \text{with} \quad  \tilde{h} (\Theta)=\int_0^\Theta h'(s) \ds, \quad h(0)=c^2(\Thetaa).
\end{equation}
So henceforth, we consider system \eqref{coupled_problem_eq_Cattaneo} in the following form  
\begin{subequations}\label{Main_system} 
\begin{equation}
\left\{ \label{modified_temp_eq}
\begin{aligned}
&(1- 2k (\Theta) p)p_{tt}-h (0)\Delta p - b \Delta p_t = 2k (\Theta) (p_{t})^2+\tilde{h} (\Theta)\Delta p, \quad &\text{in} \ \Omega \times (0,T),\\
&\tau m\Theta_{tt} +(m +\tau \ell) \Theta_t + \ell \Theta -\kappaa \Delta \Theta = \mathcal{Q}(p_t) + \tau \partial_t \mathcal{Q}(p_t), \quad &\text{in} \ \Omega \times (0,T).
\end{aligned}
\right.
\end{equation}   
Note that the purpose of the shift in the temperature variable is that now we have homogeneous boundary conditions for both the pressure and the temperature
\begin{eqnarray} \label{homog_dirichlet}
p\vert_{\partial \Om}=0, \qquad \Theta\vert_{\partial \Om}=0
\end{eqnarray} 
and the initial conditions are given by
\begin{eqnarray} \label{init_cond}
p|_{t=0}=p_0,\quad p_t|_{t=0}=p_1,\quad \Theta|_{t=0}=\Theta_0:=\bar{\Theta}_0-\Thetaa,\quad \Theta_t|_{t=0}= \Theta_1.
\end{eqnarray}
\end{subequations}
The function $\mathcal{Q}$ is given by \eqref{Q_definition}; however our proofs also work for quite general $\mathcal{Q}(p_t)$ satisfying  Assumption 2 in \cite{Nikolic_2022}. Further, the dependence of the acoustic parameters $c$ and $K$ on the temperature is assumed to be polynomial, in agreement with what is typically considered in the literature, see \cite{bilaniuk1993speed}.

\textbf{Notation.} Throughout the paper, we assume that $\Omega \subset \R^d$, where $d \in \{1,2,3\}$, is a bounded, smooth domain with a $C^3$ boundary. We denote by $T>0$ the final propagation time. The letter $C$ denotes a generic positive constant
that does not depend on time, and can have different values on different occasions.  
We often write $f \lesssim g$ where there exists a constant $C>0$, independent of parameters of interest such that $f\leq C g$. 
We often omit the spatial and temporal domain when writing norms; for example, $\|\cdot\|_{L^p L^q}$ denotes the norm in $L^p(0,T; L^q(\Omega))$.

To formulate our results, we require the following assumptions on the medium parameters $h$ and $k$.

\begin{assumption} \label{Assumption1}
We assume  that $h \in C^2(\mathbb{R})$ and  there exists $h_1>0$ such that
\begin{subequations}
\begin{equation} \label{bound_h}
h (s) \geq h_1, \quad \forall s \in \mathbb{R}.\tag{H1}
\end{equation}
Moreover, assume that there exist $\gamma_1 >0$ and $C>0$, such that
\begin{equation} \label{h''_assump}
\vert h ''(s) \vert \leq C (1+\vert s \vert^{\gamma_1}), \quad \forall s \in \mathbb{R}.\tag{H2}
\end{equation} 
Using Taylor's formula, we also have
\begin{equation} \label{h'_assump}
\vert h '(s) \vert \leq C (1+\vert s \vert^{1+\gamma_1}), \quad \forall s \in \mathbb{R}.\tag{H3}
\end{equation}
\end{subequations}
Since the function $k $ is related to the speed of sound by the formula \eqref{funct_k}, it follows that
\begin{subequations}
\begin{equation}\label{k_1}
\vert k (s) \vert \leq k_1:=\frac{\beta }{\rho h_1}.\tag{K1}
\end{equation}
Further, we have 
\begin{equation}
\begin{aligned}
\vert k ''(s) \vert &\lesssim k_1^2 \vert h ''(s) \vert+k_1^3 \vert h '(s) \vert^2 \lesssim  k_1^2(1+\vert s \vert^{\gamma_1})+k_1^3(1+\vert s \vert^{1+\gamma_1})^2,
\end{aligned}
\end{equation}
which by using Taylor's formula, implies that there exists $\gamma_2>0$, such that
\begin{equation} \label{properties_k}
\vert k '(s) \vert \lesssim (1+\vert s \vert^{1+\gamma_2}), \qquad \vert k ''(s) \vert \lesssim (1+\vert s \vert^{\gamma_2}).\tag{K2}
\end{equation}
\end{subequations}
\end{assumption}
We also assume the initial data in \eqref{init_cond} to fulfill the following regularity and compatibility conditions.
\begin{assumption} \label{Assumption_compatibility} 
 Let the initial data satisfy
 \begin{equation}   
\begin{aligned}
(p_0, p_1) & \in \big[H^3(\Om)\cap H^1_0(\Om)\big]\times \big[H^3(\Om) \cap H^1_0(\Om)\big],\\  
(\Theta_0, \Theta_1) & \in \big[H^2(\Om) \cap H^1_0(\Om)\big] \times H^1_0(\Om),
\end{aligned}
\end{equation}
such that $1-2k(\Theta_0) p_0$ does not degenerate. We also assume  the compatibility conditions:
\begin{equation}
p_2 \in H^1_0(\Om), \quad \Theta_2 \in L^2(\Om) 
\end{equation}
where $p_2:= \partial_t^2 p(0,x), \Theta_2:=\partial_t^2 \Theta(0,x), x \in \Omega, k=1, 2$ are defined formally and recursively in terms of $p_0, p_1, \Theta_0, \Theta_1$ from the equations \eqref{modified_temp_eq} as follows
\begin{equation}
\begin{aligned}
&(1-2k(\Theta_0) p_0)p_{tt}(0)=h(\Theta_0) \Delta p_0+b \Delta p_1 +k(\Theta_0)p_1^2; \\
& \tau m \Theta_{tt}(0)= -(m +\tau \ell) \Theta_1- \ell \Theta_0 + \kappaa \Delta \Theta_0+\frac{2b}{\rhoa c_{a}^4} (p_1)^2+\frac{4b}{\rhoa c_{a}^4} p_1 p_2.
\end{aligned}
\end{equation}
\end{assumption}
\subsection{Local well-posedness of the Westervelt-Pennes-Cattaneo system}
We recall here a result on the existence of local-in-time solutions to system \eqref{Main_system}, which was proven in \cite{Benabbas_Said_Houar_2023}. Let $X_p, X_\Theta$ denote the following spaces of solutions 
\begin{equation}\label{Functional_Spaces}
\begin{aligned}
X_p=&\,\Big\{p \in L^{\infty}(0, T; H^3( \Om) \cap H^1_0(\Om)), \\
&  \quad p_t \in L^{\infty}(0, T; H^2( \Om) \cap H^1_0(\Om))\cap L^{2}(0, T; H^3( \Om) \cap H^1_0(\Om)),\\
& \quad p_{tt} \in L^{\infty}(0, T; H^1_0(\Om))\cap L^{2}(0, T; H^2( \Om) \cap H^1_0(\Om)),\\
& \quad p_{ttt} \in L^{2}(0, T; L^2(\Om)) \Big\};\\
X_\Theta=&\, \{ \Theta \in L^\infty (0, T;  H^2(\Om) \cap H^1_0(\Om)), \Theta_t \in L^\infty (0, T; H^1_0(\Om)),\\
& \quad \Theta_{tt} \in L^\infty(0, T; L^2(\Om))\}.
\end{aligned}
\end{equation}
Then, under Assumptions \ref{Assumption1} and  \ref{Assumption_compatibility}, we have local well-posedness in $X_p \times X_\Theta$, which is uniform with respect to the relaxation parameter $\tau$, see \cite{Benabbas_Said_Houar_2023}.     
 \begin{theorem} \label{wellposedness_thm}
Let $T>0$ and $\tau>0$ be a fixed small constant. Let Assumptions \ref{Assumption1}, \ref{Assumption_compatibility} hold and assume that 
\begin{equation}
\|p_0\|_{H^3}+\|p_1\|_{H^2}+\|p_{2}\|_{H^1}\leq R_1.
\end{equation}
Then, there exists $\delta_1=\delta_1(T, R_1)>0$, small enough, such that if  
\begin{equation}
\|p_0\|_{H^2}+\|p_1\|_{H^1}+\|p_{2}\|_{L^2}\leq \delta,
\end{equation}
then system \eqref{Main_system} has a unique solution $(p, \Theta) \in X_p \times X_\Theta$. 
\end{theorem}
 The proof of Theorem \ref{wellposedness_thm} in \cite{Benabbas_Said_Houar_2023} relied on   assuming small values of $R_1$. However, it can be refined by requiring  only $\delta$ to be small, allowing  $R_1$ to be arbitrarily large. 
\section{Main results}\label{Section_Main_Result}
In this section, we present the main results of this paper and explain the strategy of the proof. The global existence result is stated in Theorem \ref{Them_Global_solutions}, while the decay rate  is contained in Theorem \ref{Them_exponential_decay}. The proof of these results will be given in  Section \ref{Sec_Global Existence}. 
\begin{theorem}[\textbf{Global existence}] \label{Them_Global_solutions}
Let the initial data and the medium parameters satisfy Assumptions \ref{Assumption1} and  \ref{Assumption_compatibility}. Let $M_0>0$ be such that  
\begin{equation} \label{initial_data_boundedness}
\|p_0 \|_{H^3}^2 +\| p_1 \|_{H^2}^2+\| p_2 \|_{H^1}^2+ \|\Theta_0 \|_{H^2}^2 +\| \Theta_1 \|_{H^1}^2 \leq M_0.
\end{equation}
Then there exists $\eta_0= \eta_0(M_0)>0$ sufficiently small, such that if 
\begin{equation} \label{smallness_initial_data}
\|p_0 \|_{H^2}^2 +\| p_1 \|_{H^1}^2+\| p_2 \|_{L^2}^2+ \|\Theta_0 \|_{H^1}^2 +\| \Theta_1 \|_{L^2}^2 \leq  \eta_0,
\end{equation}
the system \eqref{Main_system} admits a unique global solution $(p, \Theta) \in X_p \times X_\Theta$, satisfying for all $t \in [0, \infty)$
\begin{equation}
\begin{aligned}
&\| \Delta p(t) \|^2_{L^2} + \| \nabla p_t(t) \|^2_{L^2}+ \| p_{tt}(t) \|^2_{L^2} \leq C_1(M_0) \eta_0 ,  \quad  \text{and}\\
& \| p(t) \|^2_{H^3} + \| p_t(t) \|^2_{H^2}+ \| p_{tt}(t) \|^2_{H^1} \leq C_2 M_0.
\end{aligned} 
\end{equation}
\end{theorem}
\begin{theorem}[\textbf{Decay rate}] \label{Them_exponential_decay} Let $(p,\Theta)$ be the global-in-time solution to \eqref{Main_system} provided in Theorem \ref{Them_Global_solutions}. Then under the assumptions of Theorem \ref{Them_Global_solutions}, there exist $\tilde{C}= \tilde{C}(M_0, \eta_0)$ and $\omega >0$ such that the stability estimate
\begin{equation}
\begin{aligned}
& \| p(t) \|^2_{H^3} + \| p_t(t) \|^2_{H^2}+ \| p_{tt}(t) \|^2_{H^1} \leq \tilde{C}(M_0, \eta_0) e^{-\omega t}
\end{aligned} 
\end{equation}
holds for all $t>0$.
\end{theorem}
\subsection{Discussion of the main result}\label{Sec:Discussion}
 Before moving onto the proof, we briefly discuss the statements made above in Theorems \ref{Them_Global_solutions} and \ref{Them_exponential_decay}. 
 \begin{enumerate}
\item[1.] Similarly to the result in \cite{bongarti2021vanishing} (see also \cite{Said-Houari_2022_CDS}), we only assume  the lower-order  Sobolev norms of the initial data to be small, while the higher-order norms can be arbitrarily large. To do this, we do not rely directly on the embeddings   $H^2\hookrightarrow L^\infty, H^1\hookrightarrow L^4$, but we instead use the interpolation inequalities \eqref{lady}, \eqref{Agmon}, which give the flexibility to impose smallness on lower-order norms   
when we estimate the nonlinear terms. 
\item[2.] Among the difficulties encountered when trying to prove our result stems from the fact that the two equations in \eqref{modified_temp_eq} are coupled via the temperature in the sound coefficients of the Westervelt equation and via the source term in the hyperbolic Pennes' equation.  This differs from the thermoelastic systems \cite{racke2000evolution}, in which the coupling is within the linearized system, which generates dissipation for both components of the solution (even when the elastic component is not directly damped).  Hence, the presence of the term $-b\Delta p_t,\, b>0$ is crucial in our analysis.  It is an interesting open problem to show a global existence in the case $b=0$.     
\item[3.] The hyperbolic nature of the Pennes' equation \eqref{Hyperbolic_Pennes} introduces a higher degree of complexity in the analysis compared to the parabolic Pennes' equation  \cite{NIKOLIC2022628}; first, it does not yield any smoothing properties of the equation of the temperature and second it produces the source term $\tau \partial_t\mathcal{Q}(p_t)$, which requires more regularity assumption on the pressure component to control it.
\end{enumerate}



\subsection{Some useful inequalities and embedding results} In preparation for the upcoming analysis, we invoke some theoretical results that we shall frequently use.    
\subsection*{Sobolev embeddings}
Among the main helpful tools are Sobolev embeddings, especially the continuous embeddings $H^1(\Om) \hookrightarrow L^4(\Om)$ and $H^2(\Om) \hookrightarrow L^\infty(\Om)$.  In particular, using Poincar\'{e}'s inequality we obtain for $v \in H^1_0(\Om)$ (see \cite[Theorem 7.18]{salsa2016partial})
\begin{equation}\label{Sobolev_Embedding}
\begin{aligned}
& \text{if} \ d>2, \quad \Vert v \Vert_{L^p} \leq C \Vert \nabla v \Vert_{L^2}  \quad \text{for} \quad  2 \leq p \leq \frac{2d}{d-2},\\
& \text{if} \ d=2, \quad \Vert v \Vert_{L^p} \leq C \Vert \nabla v \Vert_{L^2}  \quad \text{for} \quad  2 \leq p < \infty.\\
\end{aligned}
\end{equation}
Moreover, taking into account the boundedness of the operator $(-\Delta)^{-1}: L^2(\Om) \rightarrow H^2(\Om) \cap H^1_0(\Om)$, we find the inequality
$$\qquad \Vert v \Vert_{L^\infty} \leq C_1 \Vert v \Vert_{H^2} \leq C_2 \Vert \Delta v \Vert_{L^2}.$$ 
\subsection*{Young's inequality} We recall Young's $\varepsilon$-inequality 
\begin{equation}
xy \leq \varepsilon x^n+C(\varepsilon) y^m, \quad \text{where}\quad \ x, y >0, \quad 1 <m,n <\infty,\quad \frac{1}{m}+\frac{1}{n}=1,
\end{equation}
and $C(\varepsilon)=(\varepsilon n)^{-m/n}m^{-1}$. In particular, we will make repeated use of  the inequality  
\begin{equation}
xy\leq \varepsilon x^2+\frac{1}{4\varepsilon}y^2, \qquad \varepsilon>0. 
\end{equation}

\subsection*{Interpolation inequalities} 
Further, we will make use of  Ladyzhenskaya's inequality for $u \in H^1(\Om)$
\begin{equation} \label{lady}
\Vert u \Vert_{L^4} \leq C\Vert u \Vert_{L^2}^{1-d/4} \Vert u \Vert_{H^1}^{d/4},\qquad 1\leq d\leq 4,  
\end{equation}
and of Agmon's interpolation inequality~\cite[Ch.\ 13]{agmon2010lectures} for functions in $H^2(\Omega)$:
\begin{equation}\label{Agmon}
\| u\|_{L^\infty(\Omega)} \leq C_{\textup{A}} \|u\|_{L^2(\Omega)}^{1-d/4} \|u\|_{H^2(\Omega)}^{d/4}, \qquad d \leq 4.
\end{equation}  
We state a version of Gronwall's inequality from which the exponential decay of solutions to \eqref{Main_system} will follow. For the proof, the reader is referred to \cite[Appendix A]{Kelliher2023}.
\begin{lemma} \label{Gronwall} Let $y \in L^1_{\textup{loc}}[0, \infty)$ be a nonnegative function. Assume that there exists $\nu>0$ such that for almost every  $s \geq 0$ and for every $t \geq s$, it holds that  
\begin{equation} \label{Gronwall_assump} 
y(t) + \nu \int_s^t y(r) \text{d}r \leq y(s). 
\end{equation}
Further, assume that \eqref{Gronwall_assump} holds for $s=0$. Then, we have for all $t \geq 0$
\begin{equation}
y(t) \leq y(0) e^{-\nu t}.
\end{equation}
\end{lemma} 

\section{Energy analysis}\label{Section_3}
 
The main goal of this section is to derive uniform with respect to time energy estimates  for the solution $(p, \Theta)$ of \eqref{Main_system}. These estimates will play a crucial role in the subsequent proof of global well-posedness and the asymptotic behavior of the solution. First, we establish the estimate of the total energy related to the equation of $\Theta $ in \eqref{modified_temp_eq} and then we estimate the total energy associated with the equation of $p$ in \eqref{modified_temp_eq}.  From the technical point of view,  seeking only to assume smallness  on a lower topology makes the proof more involved since some extra estimates of lower-order energies are needed.  These estimates should be properly factored out in the nonlinear estimates using the interpolation inequalities \eqref{lady} and  \eqref{Agmon}. 

\subsection{Energy functionals}
We begin by introducing the energy functionals and the associated dissipation rates that will be used in the proofs. 
\subsubsection{Bioheat energies}
We define the energies for the bioheat equation in \eqref{modified_temp_eq}:
\begin{subequations}
\begin{equation} \label{heat_energy_1}
\begin{aligned}
\mathcal{E}_k[\Theta](t):=&\,\frac{1}{2}  (m+\ell+ \tau \ell) \| \partial_t^k\Theta(t) \|_{L^2}^2+ \frac{\tau m}{2} \| \partial_t^k\Theta_t(t) \|_{L^2}^2\\
&\qquad +\frac{\kappaa}{2} \| \nabla \partial_t^k\Theta(t) \|_{L^2}^2,\qquad k=0,1.
\end{aligned}
\end{equation}
The corresponding dissipation rates are given by 
\begin{equation} \label{heat_dissip_1}
\begin{aligned}
\mathcal{D}_k[\Theta](t) :=&\, \ell \| \partial_t^k\Theta(t) \|_{L^2}^2+ (m+ \tau \ell) \|\partial_t^k \Theta_t(t) \|_{L^2}^2\\
&\qquad +\kappaa \| \nabla \partial_t^k\Theta(t) \|_{L^2}^2, \qquad k=0,1. 
\end{aligned}
\end{equation}
\end{subequations}
Thus, the total energy for the temperature equation and the total  dissipation rate are given, respectively, as: 
\begin{subequations}  \label{heat_energy_total}
\begin{equation} \label{heat_energy_2}
\begin{aligned}
&\, \mathcal{E}[\Theta](t)= \mathcal{E}_0[\Theta](t)+\mathcal{E}_1[\Theta](t)+\tau m \| \nabla \Theta_t \|_{L^2}^2 + \kappaa \| \Delta \Theta(t) \|_{L^2}^2, 
\end{aligned}    
\end{equation}
and 
\begin{equation} \label{heat_dissiaption_energy_2}
\begin{aligned}
\mathcal{D}[\Theta](t)=&\,\mathcal{D}_0[\Theta](t)+\mathcal{D}_1[\Theta](t)
+ (m+ \tau \ell) \| \nabla \Theta_t(t) \|_{L^2}^2+\kappaa \| \Delta \Theta \|^2_{L^2}. 
\end{aligned}    
\end{equation}
\end{subequations}
\subsubsection{Acoustic energies}
The acoustic energies are defined as 
\begin{equation}
    \begin{aligned}
        E_1[p, \Theta](t)&:=\frac{1}{2} \big(\Vert  \sqrt{1-2 k( \Theta) p} p_{t t}(t) \Vert^2_{L^2}+ (1+h (0)) \Vert \nabla p_t(t) \Vert^2_{L^2} \\
        &\,\qquad \qquad + (b+h (0))\Vert \Delta p(t) \Vert^2_{L^2} \big),\\   
        E_2[p](t)&:=\frac{1}{2} \big((1+b)\Vert  \nabla p_{tt}(t) \Vert_{L^2}^2+b \Vert  \nabla \Delta p(t) \Vert^2_{L^2} +h(0) \| \Delta p_t \|_{L^2}^2\big).\\
    \end{aligned}
\end{equation}
We denote by  $D_i[p], i=1, 2$  the corresponding dissipation rates given by 
\begin{equation}\label{dissipation}
\begin{aligned}
D_1[p](t)&\,:= b\Vert \nabla p_{tt}(t) \Vert^2_{L^2} +     \| \sqrt{h(0)} \Delta p (t)\|_{L^2}^2 + b \| \Delta p_t (t) \|_{L^2}^2,\\
D_2[p](t)&\,:= \Vert \sqrt{h(0)} \nabla \Delta p(t) \Vert^2_{L^2} + b \Vert \Delta p_{tt}(t) \Vert^2_{L^2}+\| p_{ttt}(t)  \|^2_{L^2}.
\end{aligned} 
\end{equation}
As we will be applying a continuity argument to obtain global a priori bounds for the local solution $(p,\Theta) \in X_p \times X_\Theta$, we make the following a priori assumption \cite{Tao2006NonlinearDE}  
\begin{equation}
2 \|k(\Theta) p\|_{L^\infty L^\infty}\leq \mathsf{m} <1, 
\end{equation}
with $ \mathsf{m}>0$ independent of $t$. This then yields that there exist positive constants  $0<\alpha_1 <\alpha_2 $ independent of time such that 
\begin{equation} \label{degeneracy}
\alpha_1 \leq 1-2 k(\Theta) p \leq \alpha_2.
\end{equation}
This is necessary to avoid degeneracy of the pressure wave equation and  persists for all times as long as we impose a smallness condition on a lower-order norm of the initial data, see  \cite{kaltenbacher2009global, LasieckaOng} and Section \ref{Section_Proof_Global} below.
\subsection{The bioheat equation}   
In what follows, we will show that, for the local-in-time solution to the temperature equation in \eqref{Main_system} $\Theta \in X_\Theta$, the energies $\mathcal{E}_0[\Theta], \mathcal{E}_1[\Theta]$ and $\mathcal{E}[\Theta]$ satisfy a priori  estimates that are uniform with respect to $t$. In fact, we have a  global bound for the total energy $\mathcal{E}[\Theta]$, which is stated in the proposition below.
\begin{proposition} \label{prop_theta} Let $\tau >0$. For all $t \geq 0$, the solution $(\Theta, p) \in X_p \times X_\Theta$ satisfies 
\begin{equation} \label{main_estimate_energy_theta}
\begin{aligned}
&\, \mathcal{E}[\Theta](t)+ \int_0^t \mathcal{D}[\Theta](s) \ds \\
 \lesssim &\,\mathcal{E}[\Theta](0)+ \int_0^t   \Big( E_1[p]+ (E_1[p])^{1- \frac{d}{4}} (E_2[p])^{ \frac{d}{4}} \Big)\big( D_1[p] + D_2[p]\big)   \ds 
\end{aligned}
\end{equation}
where the hidden constant does not depend on $t$.
\end{proposition}
The proof of Proposition \ref{prop_theta} will be given later, and it is the result of Lemmas \ref{lemma1} and \ref{lemma2}.  
\begin{lemma} \label{lemma1} Let $\tau>0$. It holds that 
\begin{equation} \label{estimate_energy_1}
\begin{aligned}
 \mathcal{E}_0[\Theta](t)+ \int_0^t \mathcal{D}_0 [\Theta](s) \ds \lesssim &\, \mathcal{E}_0[\Theta](0)+ \int_0^t E_1[p](s) D_1[p](s) \ds 
\end{aligned}
\end{equation}
for all $t \geq 0$.

\end{lemma}
\begin{proof}
We multiply the bioheat equation \eqref{bioheat_eq} by $\Theta_t$ and integrate over space,  using integration by parts, we obtain the identity
\begin{equation} 
\begin{aligned}
\frac{1}{2}\ddt \big(\tau m \| \Theta_t \|_{L^2}^2+\ell \| \Theta \|_{L^2}^2 &\,+\kappaa \| \nabla \Theta \|_{L^2}^2 \big)+ (m +\tau \ell) \| \Theta_t \|_{L^2}^2\\
&\,= \intO (\mathcal{Q}(p_t) +\tau \partial_t \mathcal{Q}(p_t)) \Theta_t \dx.
\end{aligned}
\end{equation}
Applying H\"{o}lder and Young inequalities, we get
\begin{equation} \label{ineq1}
\begin{aligned}
&\, \ddt \big(\tau m \| \Theta_t \|_{L^2}^2+\ell \| \Theta \|_{L^2}^2  +\kappaa \| \nabla \Theta \|_{L^2}^2 \big)+ (m +\tau \ell) \| \Theta_t \|_{L^2}^2\\
 \leq &\, \frac{1}{2 m}\big(\| \mathcal{Q}(p_t) \|_{L^2}^2 +\tau^2 \| \partial_t \mathcal{Q}(p_t) \|_{L^2}^2 \big)+ \frac{m}{2}\| \Theta_t \|_{L^2}^2.
\end{aligned}
\end{equation}
The last term on the right-hand side can be absorbed by the dissipative term on the left. Thus, integrating with respect to time, it follows that
\begin{equation} \label{energy_1_est_1}
\begin{aligned}
&\, \tau m \| \Theta_t(t) \|_{L^2}^2+\ell \| \Theta (t)\|_{L^2}^2  +\kappaa \| \nabla \Theta(t)  \|_{L^2}^2 + \int_0^t(\frac{m}{2} +\tau \ell) \| \Theta_t(s) \|_{L^2}^2 \ds\\
 \leq &\, \mathcal{E}_0[\Theta](0)+ \frac{1}{2 m} \int_0^t \big(\| \mathcal{Q}(p_t) \|_{L^2}^2 +\tau^2 \| \partial_t \mathcal{Q}(p_t) \|_{L^2}^2 \big) \ds. 
\end{aligned}
\end{equation}
Next, we test the equation \eqref{bioheat_eq} by $\Theta$ and  integrate with respect to $x$, using  integrating by parts, we find
\begin{equation} \label{identity_theta}
\begin{aligned}
&\, \frac{m+ \tau \ell}{2}\ddt  \| \Theta \|_{L^2}^2 + \ell \| \Theta \|_{L^2}^2+\kappaa \| \nabla \Theta \|_{L^2}^2\\
=&\, -\tau m \ddt\left( \intO \Theta_{t} \Theta \dx\right)+ \tau m \| \Theta_t \|_{L^2}^2+ \intO (\mathcal{Q}(p_t) +\tau \partial_t \mathcal{Q}(p_t)) \Theta \dx
\end{aligned}
\end{equation}
where we have relied on the relation
\begin{equation}
\tau m \ddt \left( \intO \Theta_{t} \Theta \dx\right)= \tau m \intO \Theta_{tt} \Theta \dx+ \tau m \intO \vert \Theta_{t} \vert^2 \dx. 
\end{equation}
Then, we can estimate the right-hand side of \eqref{identity_theta}  as follows
\begin{equation} \label{ineq2}
\begin{aligned}
&\, \frac{m+ \tau \ell}{2}\ddt \| \Theta \|_{L^2}^2 + \ell \| \Theta \|_{L^2}^2+\kappaa \| \nabla \Theta \|_{L^2}^2\\
\leq &\,   -\tau m \ddt\left( \intO \Theta_{t} \Theta \dx\right)+ \tau m \| \Theta_t \|_{L^2}^2 \\
&\,+ \frac{1}{2\ell} \big( \| \mathcal{Q}(p_t) \|_{L^2}^2 +\tau^2 \| \partial_t \mathcal{Q}(p_t) \|_{L^2}^2 \big)+ \frac{\ell}{2}\| \Theta \|_{L^2}^2.
\end{aligned}
\end{equation}
Observe that the last term on the right can be absorbed by the dissipation on the left of \eqref{ineq2}. Further, integrating from $0$ to $t$, we obtain  
\begin{equation}  
\begin{aligned}
&\, \frac{m+ \tau \ell}{2}\| \Theta(t) \|_{L^2}^2 + \int_0^t \Big(\frac{\ell}{2} \| \Theta (s)\|_{L^2}^2+\kappaa \| \nabla \Theta (s) \|_{L^2}^2 \Big) \ds\\
\leq &\,  \frac{m+ \tau \ell}{2}\| \Theta_0 \|_{L^2}^2+  \tau m \intO \vert \Theta_{t}(t)  \Theta (t)  \vert \dx+ \tau m \intO \vert \Theta_{1}  \Theta_0  \vert \dx \\
&\, +  \int_0^t \tau m \| \Theta_t(s) \|_{L^2}^2 \ds+ \frac{1}{2\ell} \int_0^t \big( \| \mathcal{Q}(p_t) \|_{L^2}^2 +\tau^2 \| \partial_t \mathcal{Q}(p_t) \|_{L^2}^2 \big) \ds.
\end{aligned}
\end{equation}
Now,  using Young's inequality gives the bound
\begin{equation}  
\begin{aligned}
&\, \frac{m+ \tau \ell}{2}\| \Theta(t) \|_{L^2}^2 + \int_0^t \Big(\frac{\ell}{2} \| \Theta (s)\|_{L^2}^2+\kappaa \| \nabla \Theta (s) \|_{L^2}^2 \Big) \ds\\
\leq &\, \frac{m+ \tau \ell}{2} \| \Theta_0 \|_{L^2}^2+  \frac{\tau m^2}{\ell} \Vert \Theta_{t}(t) \Vert_{L^2}^2+ \frac{\tau \ell}{4}\Vert  \Theta(t)  \Vert_{L^2}^2 + \frac{\tau m}{2}\big(  \Vert \Theta_{1} \Vert_{L^2}^2+ \Vert  \Theta_0  \Vert_{L^2}^2 \big)\\
&\,+ \int_0^t \tau m \| \Theta_t(s) \|_{L^2}^2 \ds+ \frac{1}{2\ell} \int_0^t \big( \| \mathcal{Q}(p_t) \|_{L^2}^2 +\tau^2 \| \partial_t \mathcal{Q}(p_t) \|_{L^2}^2 \big) \ds.
\end{aligned}
\end{equation}
Clearly, the third term on the right of the above estimate can be absorbed by the first term on the left, so that we obtain
\begin{equation} \label{energy_1_est_2}
\begin{aligned}
&\, \big(\frac{m}{2}+\frac{\tau \ell}{4} \big)\| \Theta(t) \|_{L^2}^2 + \int_0^t \Big( \frac{\ell}{2} \| \Theta (s)\|_{L^2}^2+\kappaa \| \nabla \Theta (s) \|_{L^2}^2 \Big) \ds\\ 
\leq &\, \frac{m+ \tau (m+\ell)}{2} \big(  \Vert \Theta_{0} \Vert_{L^2}^2+ \Vert  \Theta_1  \Vert_{L^2}^2 \big)+  \frac{\tau m^2}{\ell} \Vert \Theta_{t}(t) \Vert_{L^2}^2 + \int_0^t \tau m \| \Theta_t(s) \|_{L^2}^2 \ds\\
&\,+ \frac{1}{2\ell} \int_0^t \big( \| \mathcal{Q}(p_t) \|_{L^2}^2 +\tau^2 \| \partial_t \mathcal{Q}(p_t) \|_{L^2}^2 \big) \ds.
\end{aligned}
\end{equation}
In order to absorb  the second and third term on the right-hand side of \eqref{energy_1_est_2}, we sum up the estimates \eqref{energy_1_est_1} and $\frac{\ell}{2m} \times$\eqref{energy_1_est_2},  then we have for all $t \geq 0$  
\begin{equation} \label{estimate_theta_1}
\begin{aligned}
&\, \frac{\ell}{2m} \big(\frac{m}{2}+\frac{\tau \ell}{4} \big)\| \Theta(t) \|_{L^2}^2 + \ell \| \Theta (t)\|_{L^2}^2 + \frac{\tau m}{2} \| \Theta_t(t) \|_{L^2}^2+\kappaa \| \nabla \Theta(t)  \|_{L^2}^2 \\
&\, + \frac{\ell}{2m} \int_0^t \Big( \frac{\ell}{2} \| \Theta (s)\|_{L^2}^2+\kappaa \| \nabla \Theta (s) \|_{L^2}^2 \Big) \ds+ \int_0^t(\frac{m}{2} +\frac{\tau \ell}{2}) \| \Theta_t(s) \|_{L^2}^2 \ds\\ 
\lesssim &\,  \mathcal{E}_0[\Theta](0)+  \int_0^t  \big( \| \mathcal{Q}(p_t) \|_{L^2}^2 + \| \partial_t \mathcal{Q}(p_t) \|_{L^2}^2 \big) \ds.
\end{aligned}
\end{equation}

Recalling the formula for $\mathcal{Q}$ in  \eqref{Q_definition}, we have
\begin{equation}
\begin{aligned}
 \| \mathcal{Q}(p_t) \|_{L^2}^2 + \| \partial_t \mathcal{Q}(p_t) \|_{L^2}^2 \lesssim &\, \| (p_t)^2 \|_{L^2}^2 + \| p_t p_{tt} \|_{L^2}^2 \\
 \lesssim  &\, \| p_t \|_{L^4}^2 \| p_t \|_{L^4}^2 + \| p_t \|_{L^4}^2 \| p_{tt} \|_{L^4}^2\\
 \lesssim  &\, \| \nabla p_t \|_{L^2}^2 ( \| \Delta p_t \|_{L^2}^2 +   \| \nabla p_{tt} \|_{L^2}^2 )
\end{aligned}
\end{equation}
where we have used the embedding $H^1(\Om) \hookrightarrow L^4(\Om)$ and the inequalities 
\begin{equation}
\| p_t \|_{H^1} \lesssim \| p_t \|_{H^2} \lesssim \| \Delta p_t \|_{L^2}.
\end{equation}
Integrating over $(0, t)$, we obtain
\begin{equation} \label{estimate_Q}
\begin{aligned}
\int_0^t \big( \| \mathcal{Q}(p_t) \|_{L^2}^2 + \| \partial_t \mathcal{Q}(p_t) \|_{L^2}^2 \big) \ds \lesssim &\, \int_0^t E_1[p](s) D_1[p](s) \ds.
\end{aligned}
\end{equation}
Thus, combining \eqref{estimate_theta_1} and \eqref{estimate_Q}, we end up with estimate \eqref{estimate_energy_1}, thereby concluding the proof of Lemma \ref{lemma1}.
\end{proof}

Next, we obtain some extra estimates for higher-order norms of $\Theta$ that will be useful in the sequel.  
  
\begin{lemma} \label{lemma2} Let $\tau >0$. Then for all $t \geq 0$, it holds that 
\begin{equation} \label{estimate_energy_2}
\begin{aligned}
&\, \mathcal{E}_1[\Theta](t)+ \int_0^t \mathcal{D}_1 [\Theta](s) \ds \\ 
\lesssim  &\,  \mathcal{E}_1[\Theta](0)+ \int_0^t  \Big( E_1[p]+  (E_1[p])^{1- \frac{d}{4}} (E_2[p])^{ \frac{d}{4}} \Big)\big( D_1[p] + D_2[p]\big) \ds .
\end{aligned}
\end{equation}
The hidden constant in \eqref{estimate_energy_2} does not depend on $t$.
\end{lemma}
\begin{proof} By taking the time derivative of the equation  \eqref{bioheat_eq}, we can see that the function $\tilde{\Theta}:= \Theta_t$ satisfies in $L^2(0, T; H^{-1}(\Om))$ the same wave equation as $\Theta$ with a different source term
\begin{equation} \label{theta_t_eq}
\tau m \tilde{\Theta}_{tt} +(m + \tau \ell) \tilde{\Theta}_t +\ell \tilde{\Theta}- \kappaa \Delta \tilde{\Theta}= \partial_t \mathcal{Q}(p_t) +\tau \partial_t^2 \mathcal{Q}(p_t). 
\end{equation}
Thus to establish the bound \eqref{estimate_energy_2}, we test the equation \eqref{theta_t_eq} by  $\tilde{\Theta}_t, \tilde{\Theta}$ and go through exactly the same steps as in the proof of Lemma \ref{lemma1}, only the source term changes. This will lead to a similar estimate to \eqref{estimate_theta_1} for the function $\tilde{\Theta}$
\begin{equation} \label{estimate_theta_tilde}
\begin{aligned}
&\, \frac{\ell}{2m} \big(\frac{m}{2}+\frac{\tau \ell}{4} \big)\| \tilde{\Theta}(t) \|_{L^2}^2 + \ell \| \tilde{\Theta} (t)\|_{L^2}^2 + \frac{\tau m}{2} \| \tilde{\Theta}_t(t) \|_{L^2}^2+\kappaa \| \nabla \tilde{\Theta}(t)  \|_{L^2}^2 \\
&\, + \frac{\ell}{2m} \int_0^t \Big(\frac{\ell}{2} \| \tilde{\Theta} (s)\|_{L^2}^2+\kappaa \| \nabla \tilde{\Theta} (s) \|_{L^2}^2 \Big) \ds+ \int_0^t\big(\frac{m}{2} +\frac{\tau \ell}{2}\big) \| \tilde{\Theta}_t(s) \|_{L^2}^2 \ds\\ 
\lesssim &\,   \mathcal{E}_0[\tilde{\Theta}](0)+  \int_0^t \big( \|  \partial_t\mathcal{Q}(p_t) \|_{L^2}^2 + \| \partial_t^2 \mathcal{Q}(p_t) \|_{L^2}^2 \big) \ds.
\end{aligned}
\end{equation}
Further, we have by using the definition of $\mathcal{Q}$ in \eqref{Q_definition}
\begin{equation}
\begin{aligned}
 \|  \partial_t\mathcal{Q}(p_t) \|_{L^2}^2 + \| \partial_t^2 \mathcal{Q}(p_t) \|_{L^2}^2 \lesssim &\, \| p_t p_{tt} \|_{L^2}^2 +\| (p_{tt})^2+ p_t p_{ttt} \|_{L^2}^2  \\
 \lesssim &\, \| p_t p_{tt} \|_{L^2}^2 +\| (p_{tt})^2 \|_{L^2}^2 + \| p_t p_{ttt} \|_{L^2}^2  \\
 \lesssim  &\, \| p_t \|_{L^4}^2 \| p_{tt} \|_{L^4}^2 + \| p_{tt} \|_{L^4}^4+ \| p_t \|_{L^\infty}^2 \| p_{ttt} \|_{L^2}^2.
\end{aligned}
\end{equation}
From here, using the interpolation inequalities \eqref{lady} and \eqref{Agmon} together with the embedding $H^1(\Om) \hookrightarrow L^4(\Om)$, elliptic regularity, Poincar\'{e}'s inequality and \eqref{degeneracy}, we get the following bound
\begin{equation}
\begin{aligned}
& \|  \partial_t\mathcal{Q}(p_t) \|_{L^2}^2 + \| \partial_t^2 \mathcal{Q}(p_t) \|_{L^2}^2\\
  \lesssim &\, \| \nabla p_t \|_{L^2}^2 \| \nabla p_{tt} \|_{L^2}^2 +  \| p_{tt} \|_{L^2}^{2-\frac{d}{2}} \| p_{tt} \|_{H^1}^{\frac{d}{2}} \| \nabla p_{tt} \|_{L^2}^2+ \| p_{t} \|_{L^2}^{2-\frac{d}{2}}  \| p_{t} \|_{H^2}^{\frac{d}{2}} \| p_{ttt} \|_{L^2}^2\\
 \lesssim &\, \big( D_1[p] + D_2[p]\big) \Big( E_1[p]+ (E_1[p])^{1- \frac{d}{4}}(E_2[p])^{ \frac{d}{4}} \Big). 
 \end{aligned}
\end{equation}
Inserting this bound into \eqref{estimate_theta_tilde} and recalling that $\tilde{\Theta}= \Theta_t$ which means $\mathcal{E}_0[\tilde{\Theta}]= \mathcal{E}_1[\Theta]$, we reach \eqref{estimate_energy_2}; thus completing the proof of Lemma \ref{lemma2}.
\end{proof}

At this point, we are ready to proceed with the proof of Proposition \ref{prop_theta}. We remark though, that with the smoothness displayed by the solution $\Theta$, we can only perform the computations below in a formal manner. To get around this, first we deal with more regular solutions, then by a density argument we acquire the same estimates for the solution $\Theta \in X_\Theta$.

\begin{proof}[Proof of Proposition \ref{prop_theta}]
We test the equation \eqref{bioheat_eq} by $-\Delta \Theta$ and we integrate over $\Om$ to find
\begin{equation}
\kappaa \| \Delta \Theta \|^2_{L^2}= \intO (-\tau m \Theta_{tt} -(m +\tau \ell) \Theta_{t}-\ell \Theta +\mathcal{Q}(p_t)+ \tau \partial_t \mathcal{Q}(p_t))(-\Delta \Theta) \dx.
\end{equation}
The right-hand side can be estimated by using Young's inequality as follows: 
\begin{equation}
\begin{aligned}
\kappaa \| \Delta \Theta \|^2_{L^2} \leq &\, \frac{1}{2 \kappaa}\Big( \ell^2 \| \Theta \|^2_{L^2}+(m+\tau \ell)^2 \| \Theta_t \|^2_{L^2} +\tau^2 m^2 \| \Theta_{tt} \|^2_{L^2} \\
&\,+ \| \mathcal{Q}(p_t) \|_{L^2}^2+\tau^2 \| \partial_t \mathcal{Q}(p_t) \|_{L^2}^2  \Big)+ \frac{\kappaa}{2} \| \Delta \Theta \|^2_{L^2},
\end{aligned}
\end{equation}
which together with \eqref{estimate_energy_1}, \eqref{estimate_energy_2} and \eqref{estimate_Q}, implies 
\begin{equation} \label{estimate_delta_theta}
\begin{aligned}
&\, \frac{\kappaa}{2}  \int_0^t \| \Delta \Theta (s) \|^2_{L^2} \ds\\ \lesssim &\, \int_0^t \big(\mathcal{D}_0[\Theta](s)+\mathcal{D}_1[\Theta](s) \big) \ds + \int_0^t  \big(\| \mathcal{Q}(p_t) \|_{L^2}^2+ \| \partial_t \mathcal{Q}(p_t) \|_{L^2}^2 \big) \ds\\
\lesssim &\,   \mathcal{E}_0[\Theta](0)+\mathcal{E}_1[\Theta](0)+\int_0^t   \Big( E_1[p]+(E_1[p])^{1- \frac{d}{4}}(E_2[p])^{ \frac{d}{4}} \Big) \big(D_1[p]+D_2[p] \big) \ds .  
\end{aligned}
\end{equation}
Next, we multiply the bioheat equation \eqref{bioheat_eq} by $-\Delta \Theta_t$ and integrate in space, using integration by parts yields the identity
\begin{equation} \label{identity_Delta_theta}
\begin{aligned}
&\, \frac{1}{2}  \ddt   \big( \tau m \| \nabla \Theta_t \|_{L^2}^2 + \kappaa \| \Delta \Theta \|_{L^2}^2 \big) + (m+ \tau \ell) \| \nabla \Theta_t \|_{L^2}^2 \\
=&\,- \ell \intO \nabla \Theta \cdot \nabla \Theta_t \dx +\intO  \nabla ( \mathcal{Q}(p_t)+\tau \partial_t \mathcal{Q}(p_t)) \cdot \nabla \Theta_t.
\end{aligned}
\end{equation}
We note that the available regularity for the local-in-time solution $(p, \Theta) \in X_p \times X_\Theta$ makes the computations leading to the identity \eqref{identity_Delta_theta} and the estimates below, only formal. However, this could be justified by first establishing the bounds for smoother solutions to the equation  \eqref{bioheat_eq}, whose existence is provided in \cite[Theorem 6, p. 391]{evans2010partial}. Hence, by an approximation argument like the one presented in \cite[Proposition 2.1]{LasieckaTataru}, we can recover the same estimate \eqref{delta_theta_est_1} below for the solution $\Theta \in X_\Theta$.\\
Applying Young's inequality to the right-hand side of \eqref{identity_Delta_theta}, we obtain
\begin{equation} \label{delta_theta_est}
\begin{aligned}
&\, \frac{1}{2}  \ddt   \big( \tau m \| \nabla \Theta_t \|_{L^2}^2 + \kappaa \| \Delta \Theta \|_{L^2}^2 \big) + (m+ \tau \ell) \| \nabla \Theta_t \|_{L^2}^2 \\
\leq &\, \frac{\ell^2}{m}  \| \nabla \Theta \|_{L^2}^2  + \frac{1}{m} \| \nabla ( \mathcal{Q}(p_t)+\tau \partial_t \mathcal{Q}(p_t)) \|_{L^2}^2 + \frac{m}{2} \| \nabla \Theta_t \|_{L^2}^2.
\end{aligned}
\end{equation} 
The second term on the right-hand side can be bounded as follows
\begin{equation}
\begin{aligned}
 \| \nabla ( \mathcal{Q}(p_t)+\tau \partial_t \mathcal{Q}(p_t)) \|_{L^2}^2 \lesssim &\,   \| \nabla ( \mathcal{Q}(p_t) ) \|_{L^2}^2 + \| \nabla ( \partial_t \mathcal{Q}(p_t)) \|_{L^2}^2 \\
\lesssim &\, \| 2 p_t \nabla p_t \|_{L^2}^2 + \| 2 p_{tt} \nabla p_t + 2 p_{t} \nabla p_{tt} \|_{L^2}^2  \\
\lesssim &\,  \|  p_t \nabla p_t \|_{L^2}^2 + \| p_{tt} \nabla p_t \|_{L^2}^2 + \|  p_{t} \nabla p_{tt} \|_{L^2}^2 \\
\lesssim &\,  \|  p_t  \|_{L^4}^2  \| \nabla p_t \|_{L^4}^2 + \| p_{tt} \|_{L^4}^2 \| \nabla p_t \|_{L^4}^2 + \|  p_{t} \|_{L^\infty}^2 \| \nabla p_{tt} \|_{L^2}^2 .\\
\end{aligned}
\end{equation} 
By using the interpolation inequalities \eqref{lady}, \eqref{Agmon}, the embedding $H^1(\Om) \hookrightarrow L^4(\Om)$, elliptic regularity and \eqref{degeneracy}, we can further estimate the right-hand side of the above inequality
\begin{equation} \label{est_gradient_Q}
\begin{aligned}
\| \nabla ( \mathcal{Q}(p_t)+\tau \partial_t \mathcal{Q}(p_t)) \|_{L^2}^2 \lesssim &\,   \| \nabla p_t  \|_{L^2}^{2} \| \Delta p_t \|_{L^2}^2 + \| p_{tt} \|_{L^2}^{2-\frac{d}{2}} \| \nabla p_{tt} \|_{L^2}^{\frac{d}{2}} \| \Delta p_t \|_{L^2}^2 \\
&+  \|  p_t  \|_{L^2}^{2-\frac{d}{2}} \| \Delta p_t  \|_{L^2}^{\frac{d}{2}}  \| \nabla p_{tt} \|_{L^2}^2\\
\lesssim &\, \Big( E_1[p]+ E_1[p]^{1-\frac{d}{4}} E_2[p]^{\frac{d}{4}} \Big) D_1[p].
\end{aligned}
\end{equation} 
Plugging \eqref{est_gradient_Q} into \eqref{delta_theta_est} and integrating in time, we find
\begin{equation} \label{delta_theta_est_1}
\begin{aligned}
&\, \frac{1}{2}  \big( \tau m \| \nabla \Theta_t(t) \|_{L^2}^2 + \kappaa \| \Delta \Theta(t) \|_{L^2}^2 \big) + (\frac{m}{2}+ \tau \ell)\int_0^t \| \nabla \Theta_t(s) \|_{L^2}^2 \ds \\
\lesssim &\, \mathcal{E}[\Theta](0)+ \frac{\ell^2}{m} \int_0^t \| \nabla \Theta(s) \|_{L^2}^2 \ds +  \int_0^t \Big( E_1[p]+ (E_1[p])^{1-\frac{d}{4}}  (E_2[p])^{\frac{d}{4}} \Big) D_1[p]\ds .
\end{aligned}  
\end{equation}
Now we add $\lambda \times$\eqref{delta_theta_est_1} to \eqref{estimate_energy_1} where we suitably select $\lambda>0$ in order to absorb the second term on the right of \eqref{delta_theta_est_1} by the left-hand side of \eqref{estimate_energy_1}. Then we sum up the resulting estimate,  \eqref{estimate_energy_2} and \eqref{estimate_delta_theta}. This leads to the estimate \eqref{main_estimate_energy_theta} for the total energy $\mathcal{E}[\Theta]$, which ends the proof of Proposition \ref{prop_theta}. 
\end{proof}
\subsection{The acoustic pressure equation} 
Our focus in this section  is to prove some uniform in time energy estimates for the acoustic pressure. Relying on the smoothness exhibited by the local-in-time solution $(p,\Theta) \in X_p \times X_\Theta$ and assuming \eqref{degeneracy}, we are able to show the estimate stated in the following proposition.
\begin{proposition} \label{prop_pressure} The local-in-time solution to \eqref{Main_system} $(p, \Theta)$ satisfies the estimate
\begin{equation} \label{main_est_energy_p}
\begin{aligned} 
  & E_1[p] (t) + E_2[p](t)+   \int_0^t (D_1[p](s)+D_2[p](s))  \ds  \\
\lesssim &\, E_1[p](0)+ E_2[p](0)+  \int_0^t \Big( 1 +\mathcal{E}[\Theta] + \mathcal{E}[\Theta]^{2+ \gamma_2}+\mathcal{E}[\Theta]^{ 1+ \gamma_1 } \Big) \\
&\, \qquad \times \big( E_1[p]+ (E_1[p])^{1-\frac{d}{4}}(E_2[p])^{\frac{d}{4}} +\big(\mathcal{E}_0[\Theta] \big)^{1-\frac{d}{4}} \big(\mathcal{E}[\Theta] \big)^{\frac{d}{4}} \big)\big(  D_1[p]+D_2[p] +\mathcal{D}[\Theta] \big) \ds \end{aligned}
\end{equation}
independently of $t$. 
\end{proposition}
The assertion in Proposition \ref{prop_pressure} will follow by putting together the results provided in the two upcoming lemmas, where we prove a priori bounds for $E_1[p]$ and $E_2[p]$ separately. The first of which is attained by by using the multipliers $p_{tt}, -\Delta p$ and $-\Delta p_t$. More precisely, we have the following estimate. 
 \begin{lemma} \label{lemma4} For all $t \geq 0$, it holds that
\begin{equation} \label{second_estimate_p}
\begin{aligned}
E_1 [p](t) + \int_0^t &\, D_1[p](s)\ds  \\
\lesssim  E_1[p](0)+  \int_0^t &\, \Big( 1 +\mathcal{E}[\Theta] + \mathcal{E}[\Theta]^{2+ \gamma_2}+\mathcal{E}[\Theta]^{ 1+ \gamma_1 } \Big) \\
&\times \Big( E_1 [p](s)+\big(\mathcal{E}_0[\Theta] \big)^{1-\frac{d}{4}} \big(\mathcal{E}[\Theta] \big)^{\frac{d}{4}} \Big) D_1 [p](s) \ds
\end{aligned}
\end{equation}
where the hidden constant above does not depend on $t$.
\end{lemma}
 \begin{proof}
 We differentiate with respect to time the equation \eqref{pressure_eq} to get
 \begin{equation} \label{pressure_eq_diff_t}
 \begin{aligned}
(1-2 k( \Theta) p) p_{ttt}-h (0)\Delta p_t - b \Delta p_{tt} =  
 &\, 2k' (\Theta)\Theta_t ((p_{t})^2+p p_{tt})+ 6 k (\Theta)p_t p_{tt} \\
 &\,+\tilde{h}' (\Theta) \Theta_t \Delta p +\tilde{h} (\Theta)\Delta p_t.
 \end{aligned}
 \end{equation}
 We multiply \eqref{pressure_eq_diff_t} by $p_{tt}$ and integrate over $\Om$, using  integration by parts leads to the identity
 \begin{equation} \label{identity_2}
\begin{aligned}
\frac{1}{2} &\, \ddt \Big( \Vert  \sqrt{1-2 k( \Theta) p} p_{t t} \Vert^2_{L^2}+ \Vert \sqrt{h(0)} \nabla p_t \Vert^2_{L^2}\Big)+b \Vert \nabla p_{tt} \Vert^2_{L^2}\\
= &\, \intO \Big( 2k' (\Theta) \Theta_t(p_{t})^2 p_{tt} + k' (\Theta)\Theta_t p (p_{tt})^2+ 5 k (\Theta)p_t (p_{tt} )^2 \Big) \dx  \\
 &\,+ \intO \Big(\tilde{h}' (\Theta) \Theta_t \Delta p   p_{tt} \dx  -\tilde{h}' (\Theta) \nabla \Theta \cdot \nabla p_t  p_{tt}  -  \tilde{h} (\Theta) \nabla p_t \cdot \nabla  p_{tt} \Big) \dx\\
 :=&\, R_1+R_2.
\end{aligned}
\end{equation}
First we derive a bound for $R_1$
\begin{equation}
R_1=\intO \big( 2k' (\Theta) \Theta_t(p_{t})^2 p_{tt} + k' (\Theta)\Theta_t p (p_{tt})^2+ 5 k (\Theta)p_t (p_{tt} )^2 \big) \dx.
\end{equation}
 For this, we apply H\"{o}lder's inequality to find
\begin{equation}
\begin{aligned}
\vert R_1 \vert \leq &\,    2 \| k' (\Theta) \|_{L^\infty} \| \Theta_t \|_{L^4} \| p_t \|_{L^4}^2 \| p_{tt} \|_{L^4}+ \| k' (\Theta) \|_{L^\infty} \| \Theta_t \|_{L^4} \| p \|_{L^\infty} \| p_{tt} \|_{L^2} \| p_{tt} \|_{L^4}\\
&\,+5\| k (\Theta) \|_{L^\infty} \|p_t \|_{L^4} \| p_{tt} \|_{L^2} \| p_{tt} \|_{L^4}.
\end{aligned}
\end{equation}
Using the embeddings $H^1(\Om) \hookrightarrow L^4(\Om), H^2(\Om) \hookrightarrow L^\infty(\Om)$, Young and Poincar\'{e} inequalities, elliptic regularity and \eqref{properties_k}, it follows that
\begin{equation} \label{estimate_R_1}
\begin{aligned}
\vert R_1 \vert \leq  &\,  C(\varepsilon) \Big[(1+\| \Delta \Theta \|_{L^2}^{1+\gamma_2})^2 \| \nabla \Theta_t \|_{L^2}^2 \Big( \| \nabla p_{t} \|_{L^2}^4 + \| \Delta p \|_{L^2}^2 \| \nabla p_{tt} \|_{L^2}^2 \Big)\\
&\,+ k_1   \| \nabla p_t \|_{L^2}^2 \|  \nabla p_{tt} \|_{L^2}^2  \Big]+ \varepsilon \| \nabla  p_{tt} \|_{L^2}^2\\
\lesssim &\, \Big( 1 +\mathcal{E}[\Theta] + \mathcal{E}[\Theta]^{2+ \gamma_2} \Big)E_1[p] D_1[p] +\varepsilon \| \nabla  p_{tt} \|_{L^2}^2.
\end{aligned}
\end{equation}
Similarly, we can get an estimate for $R_2$
\begin{equation}
R_2=\intO \Big(\tilde{h}' (\Theta) \Theta_t \Delta p   p_{tt} \dx  -\tilde{h}' (\Theta) \nabla \Theta \cdot \nabla p_t  p_{tt}  -  \tilde{h} (\Theta) \nabla p_t \cdot \nabla  p_{tt} \Big) \dx.
\end{equation}
Applying H\"{o}lder's inequality yields
\begin{equation}
\begin{aligned}
\vert R_2 \vert \leq &\,  \| \tilde{h}' (\Theta) \|_{L^\infty} \| \Theta_t \|_{L^4} \| \Delta p \|_{L^2} \| p_{tt} \|_{L^4}+ \| \tilde{h}' (\Theta) \|_{L^\infty}   \| \nabla \Theta \|_{L^4} \| \nabla p_t \|_{L^2} \| p_{tt} \|_{L^4} \\
&\, + \| \tilde{h} (\Theta) \|_{L^\infty}   \| \nabla p_t \|_{L^2} \| \nabla p_{tt} \|_{L^2}.
\end{aligned}
\end{equation}
Recall that
\begin{equation}
\tilde{h} (\Theta)=\int_0^\Theta h'(s) \ds,
\end{equation}
then we have the bound 
\begin{equation} \label{tilde_h_infty}
\begin{aligned}
\| \tilde{h} (\Theta) \|_{L^\infty} \leq \| \Theta \|_{L^\infty} \| h'(\Theta) \|_{L^\infty} &\, \lesssim \| \Theta \|_{L^\infty} ( 1+ \| \Theta \|_{L^\infty}^{1+ \gamma_1}) \\
&\,  \lesssim  \| \Delta \Theta \|_{ L^2} ( 1+ \| \Delta \Theta \|_{L^2}^{1+ \gamma_1})
\end{aligned}
\end{equation}
which results from \eqref{h'_assump} and elliptic regularity. Furthermore, since  $h'(\Theta)= \tilde{h}^\prime(\Theta)$, the inequality \eqref{h'_assump} gives 
\begin{equation}
\| \tilde{h}'(\Theta) \|_{L^\infty}\lesssim  1+ \| \Theta \|_{L^\infty}^{1+\gamma_1} .
\end{equation}
Hence, the above estimates along with Young's inequality yield
\begin{equation}
\begin{aligned}
\vert R_2 \vert \leq  C(\varepsilon) (1+\| \Theta \|_{L^\infty}^{1+\gamma_1})^2 \Big( &\, \| \Theta_t \|_{L^4}^2 \| \Delta p \|_{L^2}^2 +  \| \nabla \Theta \|_{L^4}^2 \| \nabla p_t \|_{L^2}^2  \\
&\, +  \| \Theta \|_{L^\infty}^2  \| \nabla p_t \|_{L^2}^2 \Big) + \varepsilon \| \nabla  p_{tt} \|_{L^2}^2 .
\end{aligned}
\end{equation}
Moreover, the interpolation inequalities \eqref{lady} and  \eqref{Agmon}  together with the embedding $H^2(\Om) \hookrightarrow L^\infty(\Om)$ and elliptic regularity leads to 
\begin{equation}
\begin{aligned}
\vert R_2 \vert \lesssim   (1+\| \Delta \Theta \|_{L^2}^{2+2\gamma_1}) \Big( &\, \| \Theta_t \|_{L^2}^{2-\frac{d}{2}}  \| \nabla  \Theta_t \|_{L^2} ^{\frac{d}{2}}\| \Delta p \|_{L^2}^2 +  \| \nabla \Theta \|_{L^2}^{2-\frac{d}{2}} \| \Delta \Theta \|_{L^2}^{\frac{d}{2}}  \| \Delta p_t \|_{L^2}^2  \\
&\, +  \| \Theta \|_{L^2}^{2-\frac{d}{2}}  \| \Delta \Theta \|_{L^2} ^{\frac{d}{2}}  \| \Delta p_t \|_{L^2}^2   \Big)+\varepsilon \| \nabla  p_{tt} \|_{L^2}^2.
\end{aligned}
\end{equation}
Observe that we have also used the following inequality to reach the estimate above
\begin{equation}
\| \nabla \Theta \|_{H^1} \leq \| \Theta \|_{H^2} \lesssim \| \Delta \Theta \|_{L^2}.
\end{equation}
Then, it follows that
\begin{equation} \label{estimate_R_2}
\begin{aligned}
\vert R_2 \vert \lesssim &\, (1+ \mathcal{E}[\Theta]^{ 1+ \gamma_1 }) \big(\mathcal{E}_0[\Theta] \big)^{1-\frac{d}{4}} \big(\mathcal{E}[\Theta] \big)^{\frac{d}{4}}D_1[p]+\varepsilon \| \nabla  p_{tt} \|_{L^2}^2.
\end{aligned}
\end{equation}
By putting together the bounds for $R_1$ and $R_2$ in \eqref{estimate_R_1} and \eqref{estimate_R_2} respectively, then  selecting $\varepsilon$ suitably small, we conclude 
\begin{equation} \label{first_estimate_E_p_2}
\begin{aligned}
\frac{1}{2} &\, \ddt \Big( \Vert  \sqrt{1-2 k( \Theta) p} p_{t t} \Vert^2_{L^2}+ \Vert \sqrt{h(0)} \nabla p_t \Vert^2_{L^2}\Big)+\frac{b}{2} \Vert \nabla p_{tt} \Vert^2_{L^2}\\
\lesssim &\, \Big( 1 +\mathcal{E}[\Theta] + \mathcal{E}[\Theta]^{2+ \gamma_2}+\mathcal{E}[\Theta]^{ 1+ \gamma_1 } \Big) \Big( E_1[p]+\big(\mathcal{E}_0[\Theta] \big)^{1-\frac{d}{4}} \big(\mathcal{E}[\Theta] \big)^{\frac{d}{4}} \Big) D_1[p] .
\end{aligned}
\end{equation}
Next, we test the equation \eqref{pressure_eq} by $-\Delta p$ and we integrate in space; thus we get the identity
\begin{equation}
\begin{aligned}
&\,\frac{b}{2} \ddt \| \Delta p \|_{L^2}^2 + h(0) \| \Delta p \|_{L^2}^2\\
= &\,\intO \big(p_{tt}  -2 k(\Theta) p p_{tt}   - 2 k(\Theta) (p_t)^2   - \tilde{h}(\Theta)  \Delta p  \big) \Delta p \dx.
\end{aligned}
\end{equation}
We estimate the right-hand side of the identity above using H\"{o}lder's inequality to obtain  
\begin{equation}
\begin{aligned}
&\,\frac{b}{2} \ddt \| \Delta p \|_{L^2}^2 + h(0) \| \Delta p \|_{L^2}^2\\
\leq  &\, \Big(\| p_{tt} \|_{L^2} + 2\| k(\Theta) \|_{L^\infty} \|  p  \|_{L^4} \| p_{tt} \|_{L^4}+  2 \| k(\Theta) \|_{L^\infty} \| p_t \|_{L^4}^2 \\
&\,+  \|  \tilde{h}(\Theta) \|_{L^\infty} \|  \Delta p \|_{L^2}\Big) \|\Delta p \|_{L^2}.
\end{aligned}
\end{equation}
Thanks to the bound \eqref{tilde_h_infty},  the embedding $H^1(\Om) \hookrightarrow L^4(\Om)$, Poincar\'{e} and Young inequalities and \eqref{properties_k}, we obtain
 \begin{equation}
\begin{aligned}
\frac{b}{2} \ddt &\, \| \Delta p \|_{L^2}^2 + h(0) \| \Delta p \|_{L^2}^2\\
\leq   C(\varepsilon) \Big( &\, \| \nabla p_{tt} \|_{L^2}^2 + k_1^2 \| \nabla p  \|_{L^2}^2 \| \nabla p_{tt} \|_{L^2}^2+  k_1^2 \| \nabla p_t \|_{L^2}^4\\
&\,+   ( 1+ \| \Theta \|_{L^\infty}^{1+ \gamma_1})^2  \| \Theta \|_{L^\infty}^2 \|  \Delta p \|_{L^2}^2 \Big) +\varepsilon \|  \Delta p \|_{L^2}^2.
\end{aligned}
\end{equation}
From here, by appealing to the interpolation inequality \eqref{Agmon} and elliptic regularity, and selecting $\varepsilon$ small enough, it results
\begin{equation} \label{second_estimate_E_p_2}
\begin{aligned}
&\,\frac{b}{2} \ddt \| \Delta p \|_{L^2}^2 + \frac{h(0)}{2} \| \Delta p \|_{L^2}^2\\
\lesssim  &\,\| \nabla p_{tt} \|_{L^2}^2 +   \| \Delta p  \|_{L^2}^2 \| \nabla p_{tt} \|_{L^2}^2+    \| \nabla p_t \|_{L^2}^2 \| \Delta  p_t \|_{L^2}^2 \\
&\,+   ( 1+ \| \Delta  \Theta \|_{L^2}^{2+2 \gamma_1})  \| \Theta \|_{L^2}^{2-\frac{d}{2}} \| \Delta \Theta \|_{L^2}^{\frac{d}{2}}  \|  \Delta p \|_{L^2}^2\\
\lesssim  &\, \| \nabla p_{tt} \|_{L^2}^2+ (1+ \mathcal{E}[\Theta]^{1+ \gamma_1})\Big( E_1[p]+ \big(\mathcal{E}_0[\Theta] \big)^{ 1-\frac{d}{4}} \big(\mathcal{E}[\Theta] \big)^{\frac{d}{4}} \Big) D_1[p].
\end{aligned}
\end{equation}
Now, we add up $\lambda \times$\eqref{second_estimate_E_p_2} to \eqref{first_estimate_E_p_2}, where $\lambda$ is suitably chosen to absorb the first term on the right-hand side of \eqref{second_estimate_E_p_2} by the left-hand side of \eqref{first_estimate_E_p_2}, we  obtain  the bound
\begin{equation} \label{intermediate_est}
\begin{aligned}
&\, \frac{1}{2} \ddt \Big( \Vert  \sqrt{1-2 k( \Theta) p} p_{t t} \Vert^2_{L^2}+ \Vert \sqrt{h(0)} \nabla p_t \Vert^2_{L^2}+ b \| \Delta p \|_{L^2}^2 \Big)\\
&\,+ \frac{b}{4} \Vert \nabla p_{tt} \Vert^2_{L^2}+\frac{h(0)}{2} \| \Delta p \|_{L^2}^2\\
\lesssim &\, \Big( 1 +\mathcal{E}[\Theta] + \mathcal{E}[\Theta]^{2+ \gamma_2}+\mathcal{E}[\Theta]^{ 1+ \gamma_1 } \Big) \Big( E_1[p]+  \big(\mathcal{E}_0[\Theta] \big)^{1-\frac{d}{4}} \big(\mathcal{E}[\Theta] \big)^{\frac{d}{4}} \Big)D_1[p].
\end{aligned}
\end{equation}
In the last step of the proof, we multiply the equation \eqref{pressure_eq} by $-\Delta p_t$ and we integrate over $\Om$,  to find 
\begin{equation}
\begin{aligned}
&\,\frac{1}{2} \ddt \Big(  \| \nabla p_t \|_{L^2}^2 + h(0)\| \Delta p \|_{L^2}^2 \Big) + b \| \Delta p_t \|_{L^2}^2\\
= &\,\intO \big(-2 k(\Theta) p p_{tt}  - 2 k(\Theta) (p_t)^2 - \tilde{h}(\Theta)   \Delta p  \big) \Delta p_t \dx. 
\end{aligned}
\end{equation}
Then, we have by using H\"older's inequality together with Young's inequality 
\begin{equation} \label{estimate_delta_p_t}
\begin{aligned}
&\,\frac{1}{2} \ddt \Big(  \| \nabla p_t \|_{L^2}^2 + h(0)\| \Delta p \|_{L^2}^2 \Big) + b \| \Delta p_t \|_{L^2}^2\\
\leq  &\, \Big(2 \| k(\Theta) \| _{L^\infty} \| p \|_{L^4} \| p_{tt} \|_{L^4} + 2 \| k(\Theta) \|_{L^\infty} \| p_t \|^2_{L^4}+  \| \tilde{h}(\Theta) \|_{L^\infty} \|   \Delta p \|_{L^2} \Big) \|  \Delta p_t \|_{L^2}\\
\leq  &\, C \Big( k_1^2 \| \nabla p \|_{L^2}^2 \| \nabla p_{tt} \|_{L^2}^2 +  k_1^2 \| \nabla p_t \|^4_{L^2}+  \| \tilde{h}(\Theta) \|_{L^\infty}^2 \|   \Delta p \|_{L^2}^2 \Big)+ \frac{b}{2}\|  \Delta p_t \|_{L^2}^2,
\end{aligned}
\end{equation}
where we have made use of \eqref{properties_k} and the embedding $H^1(\Om) \hookrightarrow  L^4(\Om)$. Moreover, recalling \eqref{tilde_h_infty} and using Ladyzhenskaya's inequality \eqref{lady}, we obtain
\begin{equation} \label{estimate_delta_p_t_1}
\begin{aligned}
&\,\frac{1}{2} \ddt \Big(  \| \nabla p_t \|_{L^2}^2 + h(0)\| \Delta p \|_{L^2}^2 \Big) + \frac{b}{2} \| \Delta p_t \|_{L^2}^2\\
\lesssim  &\, \| \Delta p \|_{L^2}^2 \| \nabla p_{tt} \|_{L^2}^2+   \| \nabla p_t \|^2_{L^2} \| \Delta p_t \|^2_{L^2}\\
&\,+ ( 1+ \| \Delta \Theta \|_{L^2}^{2+ 2 \gamma_1})\| \Theta \|_{L^2}^{2-\frac{d}{2}} \| \Delta \Theta \|_{L^2}^{\frac{d}{2}}  \|  \Delta p \|_{L^2}^2\\
\lesssim &\, \Big( 1+\mathcal{E}[\Theta]^{ 1+ \gamma_1 } \Big) \Big(E_1[p]+\big(\mathcal{E}_0[\Theta] \big)^{1-\frac{d}{4}} \big(\mathcal{E}[\Theta] \big)^{\frac{d}{4}} \Big) D_1[p].
\end{aligned}
\end{equation}
Finally, adding up the bounds \eqref{estimate_delta_p_t_1} and \eqref{intermediate_est}, then integrating in time the resulting estimate, \eqref{second_estimate_p} folllows. This completes the proof of Lemma \ref{lemma4}.
 \end{proof}

 By inspecting the already-established estimates, it becomes clear that the higher-order energy $E_2[p]$ is needed to bound some nonlinear terms. Consequently, in the subsequent lemma, our focus will be on deriving an estimate for $E_2[p]$. 
  For this, we will  differentiate the pressure equation in space as well as in time and use suitable multipliers.
 
 \begin{lemma} \label{lemma5}
 For all $t \geq 0$, we have
\begin{equation} \label{third_estimate_p}
\begin{aligned} 
   &E_2[p](t) + \int_0^t D_2[p](s) \ds \\
\lesssim &\, E_2[p](0) + \int_0^t D_1[p](s) \ds \\
 &\,+ \int_0^t  \big(1+\mathcal{E}[\Theta]+ \mathcal{E}[\Theta]^{2+\gamma_2}+ \mathcal{E}[\Theta]^{1+ \gamma_1} \big) \\
 &\,\times  \big( E_1[p] + (E_1[p])^{1-\frac{d}{4}}(E_2[p])^{\frac{d}{4}}+ (\mathcal{E}_0[\Theta])^{1-\frac{d}{4}} (\mathcal{E}[\Theta])^{\frac{d}{4}} \big) \big(D_1[p]+D_2[p] + \mathcal{D}[\Theta]  \big)  \ds,
\end{aligned}
\end{equation}
where the  hidden constant is independent of $t$.
 \end{lemma}
 \begin{proof}
 First, we apply the gradient to the pressure equation \eqref{pressure_eq},
 \begin{equation}
 \begin{aligned}
 &\nabla p_{tt}-h(0) \nabla \Delta p- \tilde{h}(\Theta) \nabla \Delta p-b \nabla \Delta p_t\\
 =&\, 2 k(\Theta)( p \nabla p_{tt} + p_{tt} \nabla p + 2 p_t \nabla p_t ) + 2 k'(\Theta) \nabla \Theta(p p_{tt}+(p_t)^2)  + \tilde{h}'(\Theta) \nabla \Theta \Delta p,  
 \end{aligned}
 \end{equation}
  multiply by $-\nabla \Delta p$ and integrate over $\Om$ to obtain
 \begin{equation} \label{identity_nabla_delta_p}
 \begin{aligned}
 &\,\frac{b}{2} \ddt \| \nabla \Delta p \|^2_{L^2} + \| \sqrt{h(0)} \nabla \Delta p \|^2_{L^2} + \big\| \sqrt{\tilde{h}(\Theta)} \nabla \Delta p \big\|^2_{L^2}\\
 =&\, \intO \nabla p_{tt} \cdot \nabla \Delta p \dx-2\intO  k(\Theta)( p \nabla p_{tt} + p_{tt} \nabla p + 2 p_t \nabla p_t ) \cdot  \nabla \Delta p \dx\\
 &\, -2 \intO  k'(\Theta)(p p_{tt}+(p_t)^2) \nabla \Theta \cdot  \nabla \Delta p \dx- \intO \tilde{h}'(\Theta) \Delta p \nabla \Theta \cdot \nabla \Delta p \dx\\
 =&\, \intO \nabla p_{tt} \cdot \nabla \Delta p \dx + R_{11}+ R_{12}+R_{13}\\
 \leq &\,  \frac{5}{4 h(0)} \| \nabla p_{tt} \|^2_{L^2} + \frac{1}{5}\| \sqrt{h(0)} \nabla \Delta p \|^2_{L^2} + \vert R_{11} \vert+ \vert R_{12} \vert+ \vert R_{13} \vert.
 \end{aligned}
 \end{equation}
 Our next goal is to estimate the terms $R_{1j}, \, j=1,2,3$ on the right-hand side of \eqref{identity_nabla_delta_p}. 
 
The integral $R_{11}$ can be estimated as follows
\begin{equation}
\begin{aligned}
\vert R_{11} \vert \leq 2 \| k(\Theta) \|_{L^\infty} \Big( \| p \|_{L^\infty} \| \nabla p_{tt} \|_{L^2} +  \| p_{tt} \|_{L^4} \| \nabla p \|_{L^4} + 2 \| p_t \|_{L^4} \| \nabla p_t \|_{L^4}\Big) \|  \nabla \Delta p \|_{L^2},
\end{aligned}
\end{equation} 
which becomes after using Young's inequality,  the embeddings $H^1(\Om) \hookrightarrow L^4(\Om)$, $H^2(\Om) \hookrightarrow L^\infty(\Om)$, elliptic regularity and \eqref{properties_k} 
\begin{equation} \label{R_1_ineq}
\begin{aligned}
\vert R_{11} \vert \leq &\, C(\varepsilon) \big(  \| \Delta p \|_{L^2}^2 \| \nabla p_{tt} \|_{L^2}^2 +  \| \nabla p_t \|_{L^2}^2 \| \Delta p_t \|_{L^2}^2\big)+ \varepsilon \|  \nabla \Delta p \|_{L^2}^2\\
\lesssim &\,  E_1[p] D_1[p] + \varepsilon \|  \nabla \Delta p \|_{L^2}^2.
\end{aligned}
\end{equation}
Likewise, we can treat the integral $R_{12}$ as 
\begin{equation}
\begin{aligned}
\vert R_{12} \vert \leq  2 \| k'(\Theta) \|_{L^\infty} \big( \|p \|_{L^\infty} \| p_{tt} \|_{L^4}+ \|p_t \|_{L^\infty} \|p_t \|_{L^4}) \| \nabla \Theta \|_{L^4} \|  \nabla \Delta p \|_{L^2}.
\end{aligned}
\end{equation}
Once again, thanks to the embeddings $H^1(\Om) \hookrightarrow L^4(\Om)$, $H^2(\Om) \hookrightarrow L^\infty(\Om)$ and \eqref{properties_k}, we infer that 
\begin{equation} \label{R_2_ineq}
\begin{aligned}
\vert R_{12} \vert \leq &\, C(\varepsilon)( 1+ \| \Delta  \Theta \|_{L^2}^{1+ \gamma_2})^2 \| \Delta \Theta \|_{L^2}^2 \big( \| \Delta p \|_{L^2}^2 \| \nabla p_{tt} \|_{L^2}^2+  \| \nabla p_t \|_{L^2}^2 \| \Delta p_t \|_{L^2}^2 \big)\\
&\,+ \varepsilon  \|  \nabla \Delta p \|_{L^2}^2\\
\lesssim &\, \big(1+ \mathcal{E}[\Theta]^{1+ \gamma_2} \big)\mathcal{E}[\Theta] E_1[p] D_1[p] + \varepsilon  \|  \nabla \Delta p \|_{L^2}^2.
\end{aligned}
\end{equation}
The last integral $R_{13}$ in \eqref{identity_nabla_delta_p} can be handled in a similar way. We have
\begin{equation} \label{ineq_R_13}
\begin{aligned}
 \vert R_{13} \vert \leq &\,  \| \tilde{h}'(\Theta) \| _{L^\infty} \|  \Delta p \|_{L^4} \| \nabla \Theta \|_{L^4}  \| \nabla \Delta p \| _{L^2} \\  
 \leq &\, C(\varepsilon)   \| \tilde{h}'(\Theta) \| _{L^\infty}^2 \|  \Delta p \|_{L^4}^2 \| \nabla \Theta \|_{L^4}^2 +\varepsilon \| \nabla \Delta p \|_{L^2}^2\\
 \lesssim &\, ( 1+ \| \Delta  \Theta \|_{L^2}^{1+ \gamma_1})^2 \|  \Delta p \|_{L^2}^{2-\frac{d}{2}} \|  \Delta p \|_{H^1}^{\frac{d}{2}} \| \Delta \Theta \|_{L^2}^2  +\varepsilon \| \nabla \Delta p \|_{L^2}^2
\end{aligned}
\end{equation}
where we have exploited the same ideas as above along with the fact that $ h'(\Theta)= \tilde{h}'(\Theta)$, \eqref{h'_assump} and the interpolation inequality \eqref{lady}. Further, thanks to the inequality
\begin{equation} 
(a+b)^\nu \leq \max \{1, 2^{\nu }\} (a^\nu +b^\nu), \quad a,b \geq 0, \ \nu>0
\end{equation}
we have 
\begin{equation} \label{inequality_laplacian}
\begin{aligned}
\| \Delta p \|_{H^1}^{\frac{d}{2}}=&\, (\| \Delta p \|_{L^2}+ \| \nabla \Delta p \|_{L^2})^{\frac{d}{2}} \lesssim \| \Delta p \|_{L^2}^{\frac{d}{2}} + \| \nabla \Delta p \|_{L^2}^{\frac{d}{2}}
\end{aligned}
\end{equation}
and by plugging \eqref{inequality_laplacian} into \eqref{ineq_R_13}, we deduce
\begin{equation} \label{R_3_ineq}
\begin{aligned}
 \vert R_{13} \vert \lesssim &\,  \big(1+ \mathcal{E}[\Theta]^{1+ \gamma_1} \big)\big( E_1[p]+ (E_1[p])^{1-\frac{d}{4}}(E_2[p])^{\frac{d}{4}}  \big)\mathcal{D}[\Theta] +\varepsilon \| \nabla \Delta p \|_{L^2}^2.
 \end{aligned}
\end{equation}
Incorporating the bounds \eqref{R_1_ineq}, \eqref{R_2_ineq} and \eqref{R_3_ineq} into \eqref{identity_nabla_delta_p} and taking $\varepsilon$ sufficiently small, it follows that
\begin{equation} \label{estimate_nabla_delta_p}
 \begin{aligned}
 \frac{b}{2} \ddt \| \nabla \Delta p \|^2_{L^2} &\, + \frac{1}{2} \| \sqrt{h(0)} \nabla \Delta p \|^2_{L^2} + \big\| \sqrt{\tilde{h}(\Theta)} \nabla \Delta p \big\|^2_{L^2}\\
 \lesssim  \| \nabla p_{tt} \|^2_{L^2}  + \big( &\,1+\mathcal{E}[\Theta]+ \mathcal{E}[\Theta]^{2+\gamma_2}+ \mathcal{E}[\Theta]^{1+ \gamma_1} \big) \\
 &\, \times \big( E_1[p]+ (E_1[p])^{1-\frac{d}{4}}(E_2[p])^{\frac{d}{4}}  \big)\big(D_1[p]+\mathcal{D}[\Theta] \big).
 \end{aligned}
\end{equation}

In the next step, we test the equation \eqref{pressure_eq_diff_t} by $p_{ttt}$ and we integrate in space, we arrive at the identity 
\begin{equation} \label{identity_p_ttt}
\begin{aligned}
&\,\frac{b}{2} \ddt \| \nabla p_{tt} \|^2_{L^2} + \| p_{ttt}  \|^2_{L^2}\\
=&\, \intO h (0)\Delta p_t p_{ttt} \dx+ \intO \Big(6 k (\Theta)p_t p_{tt}+ 2k' (\Theta)\Theta_t ((p_{t})^2+p p_{tt})+2 k( \Theta) p p_{ttt} \Big) p_{ttt} \dx\\  
 &\, + \intO \Big(\tilde{h}' (\Theta) \Theta_t \Delta p +\tilde{h} (\Theta)\Delta p_t \Big) p_{ttt} \dx\\
 =&\, \intO h (0)\Delta p_t p_{ttt} \dx+ R_{21}+ R_{22}\\
 \leq &\, C(\varepsilon)(h (0))^2 \| \Delta p_t \|_{L^2}^2 + \varepsilon \|    p_{ttt} \|_{L^2}^2+  \vert R_{21} \vert+  \vert R_{22} \vert.
\end{aligned}
\end{equation}
We need upper bounds for the integrals $R_{21}, R_{22}$. For an estimate of $R_{21}$, we begin by applying H\"{o}lder's inequality to get 
\begin{equation}
\begin{aligned}
\vert R_{21} \vert \leq &\,
 \Big(6 \| k (\Theta) \|_{L^\infty} \| p_t \|_{L^4} \| p_{tt} \|_{L^4} + 2 \| k (\Theta) \|_{L^\infty} \| p \|_{L^\infty} \| p_{ttt} \|_{L^2} \\
&\,+ 2 \| k' (\Theta) \|_{L^\infty} \| \Theta_t \|_{L^4} \big(\| p_{t} \|_{L^\infty} \| p_{t} \|_{L^4} + \|p \|_{L^\infty} \| p_{tt} \|_{L^4}  \big)   \Big) \| p_{ttt} \|_{L^2}.
\end{aligned}
\end{equation} 
By taking advantage of the properties of the function $k$ and the embeddings $H^1(\Om) \hookrightarrow L^4(\Om)$, $H^2(\Om) \hookrightarrow L^\infty(\Om)$, we find
\begin{equation}
\begin{aligned}
\vert R_{21} \vert \lesssim&\, C(\varepsilon) \Big( \, k_1^2 \| \nabla p_t \|_{L^2}^2 \| \nabla p_{tt} \|_{L^2}^2 + k_1^2 \| \Delta p \|_{L^2}^2 \| p_{ttt} \|_{L^2}^2 \\
&+( 1+ \| \Delta  \Theta \|_{L^2}^{1+ \gamma_2})^2 \| \nabla \Theta_t \|_{L^2}^2 \big(\| \Delta p_{t} \|_{L^2}^2 \| \nabla p_{t} \|_{L^2}^2 + \|\Delta p \|_{L^2}^2 \| \nabla p_{tt} \|_{L^2}^2  \big) \Big)
 +  \varepsilon \| p_{ttt} \|_{L^2}^2.
\end{aligned}
\end{equation}
This  clearly implies 
\begin{equation} \label{R21_ineq}
\begin{aligned}
\vert R_{21} \vert \lesssim &\, \big(1+  \mathcal{E}[\Theta]+ \mathcal{E}[\Theta]^{2+\gamma_2} \big)E_1[p] \big(D_1[p] + D_2[p]\big) + \varepsilon \| p_{ttt} \|_{L^2}^2.
\end{aligned}
\end{equation} 
We derive a bound for $R_{22}$ by relying on the same tools. Indeed, we have
\begin{equation}
\begin{aligned}
\vert R_{22} \vert \leq &\, \Big( \| \tilde{h}' (\Theta) \|_{L^\infty}  \| \Theta_t \|_{L^4}  \|  \Delta p \|_{L^4} + \| \tilde{h} (\Theta) \|_{L^\infty}  \| \Delta p_t  \|_{L^2} \Big) \| p_{ttt} \|_{L^2},
\end{aligned}
\end{equation}
which combined with \eqref{tilde_h_infty}, the fact that $h'(\Theta) = \tilde{h}' (\Theta)$ and \eqref{h'_assump} yields
\begin{equation}
\begin{aligned}
\vert R_{22} \vert \lesssim &\, C(\varepsilon) ( 1+ \| \Delta  \Theta \|_{L^2}^{1+ \gamma_1})^2 \Big( \| \nabla \Theta_t \|_{L^2}^2  \|  \Delta p \|_{L^4}^2 + \| \Theta \|_{L^\infty}^2 \| \Delta p_t  \|_{L^2}^2 \Big)+ \varepsilon \| p_{ttt} \|_{L^2}^2.
\end{aligned}
\end{equation}
Moreover, an application of the interpolation inequalities \eqref{lady},  \eqref{Agmon} gives
\begin{equation}
\begin{aligned}
\vert R_{22} \vert \lesssim &\, C(\varepsilon)( 1+ \| \Delta  \Theta \|_{L^2}^{2+ 2\gamma_1}) \Big(\| \nabla \Theta_t \|_{L^2}^2  \|  \Delta p \|_{L^2}^{2-\frac{d}{2}} \|  \Delta p \|_{H^1}^{\frac{d}{2}} \\
&\,+ \|  \Theta \|_{L^2}^{2-\frac{d}{2}} \| \Delta  \Theta \|_{L^2}^{\frac{d}{2}} \| \Delta p_t  \|_{L^2}^2  \Big)+ \varepsilon \| p_{ttt} \|_{L^2}^2.
\end{aligned}
\end{equation}
Exploiting the estimate \eqref{inequality_laplacian}, it results 
\begin{equation} \label{R22_ineq}
\begin{aligned}
\vert R_{22} \vert \lesssim &\,\big(1+ \mathcal{E}[\Theta]^{1+ \gamma_1} \big) \big( E_1[p] + (E_1[p])^{1-\frac{d}{4}}(E_2[p])^{\frac{d}{4}} + (\mathcal{E}_0[\Theta])^{1-\frac{d}{4}} (\mathcal{E}[\Theta])^{\frac{d}{4}}  \big)\\
&\times \big( D_1[p]+ \mathcal{D}[\Theta] \big)+ \varepsilon \| p_{ttt} \|_{L^2}^2.
\end{aligned}
\end{equation}
Inserting the estimates \eqref{R21_ineq}, \eqref{R22_ineq} into \eqref{identity_p_ttt}, keeping in mind \eqref{degeneracy}, for small enough $\varepsilon$, we get
\begin{equation} \label{estimate_p_ttt}
\begin{aligned}
&\, \frac{b}{2} \ddt \| \nabla p_{tt} \|^2_{L^2} + \frac{1}{2}\|  p_{ttt}  \|^2_{L^2}\\
  &\,\lesssim   \| \Delta p_t \|_{L^2}^2 +\Big(1+  \mathcal{E}[\Theta]+ \mathcal{E}[\Theta]^{1+ \gamma_1}+ \mathcal{E}[\Theta]^{2+\gamma_2} \Big)\\
 &\, \,\,\times  \Big( E_1[p] + (E_1[p])^{1-\frac{d}{4}}(E_2[p])^{\frac{d}{4}}+(\mathcal{E}_0[\Theta])^{1-\frac{d}{4}} (\mathcal{E}[\Theta])^{\frac{d}{4}} \Big) \big( D_1[p]+ D_2[p]+\mathcal{D}[\Theta]\big) .
\end{aligned}
\end{equation}
In order to obtain an estimate for the energy $E_3[p]$, we next multiply the times-differentiated equation \eqref{pressure_eq_diff_t} by $-\Delta p_{tt}$, integrate in space and use integration by parts to get 
\begin{equation} \label{identity_delta_p_tt}
\begin{aligned}
&\,\frac{1}{2} \ddt \Big( \| \nabla p_{tt} \|_{L^2}^2 +h(0) \| \Delta p_t \|_{L^2}^2 \Big) +b \| \Delta p_{tt} \|_{L^2}^2\\
=&\,- \intO \big(2 k( \Theta) p p_{ttt} + 2k' (\Theta)\Theta_t ((p_{t})^2+p p_{tt})+ 6 k (\Theta)p_t p_{tt} \big) \Delta p_{tt} \dx \\
 &\,- \intO \big(\tilde{h}' (\Theta) \Theta_t \Delta p +\tilde{h} (\Theta)\Delta p_t \big)\Delta p_{tt} \dx\\
 =&\, R_{31} +R_{32}.
\end{aligned}
\end{equation}
First, we estimate  $R_{31}$ as follows: 
\begin{equation}
\begin{aligned}
\vert R_{31} \vert \leq &\,   \| k( \Theta) \|_{L^\infty} \Big(2 \| p \|_{L^\infty} \| p_{ttt} \|_{L^2} +6 \|p_t \|_{L^4} \| p_{tt} \|_{L^4} \Big) \| \Delta p_{tt} \|_{L^2} \\
&\,+ 2 \|k' (\Theta) \|_{L^\infty} \| \Theta_t \|_{L^4} \Big( \| p_{t} \|_{L^\infty} \| p_{t} \|_{L^4} + \|p  \|_{L^\infty} \| p_{tt} \|_{L^4}\Big)\| \Delta p_{tt} \|_{L^2}.
\end{aligned}
\end{equation}
Further, appealing to the embeddings $H^1(\Om) \hookrightarrow L^4(\Om)$, $H^2(\Om) \hookrightarrow L^\infty(\Om)$ and \eqref{properties_k}, we find
\begin{equation}
\begin{aligned}
\vert R_{31} \vert \lesssim &\,  k_1 \Big(\| \Delta p \|_{L^2} \| p_{ttt} \|_{L^2} + \|\nabla p_t \|_{L^2} \| \nabla p_{tt} \|_{L^2} \Big)\| \Delta p_{tt} \|_{L^2} \\
&\,+ ( 1+ \| \Delta  \Theta \|_{L^2}^{1+ \gamma_2}) \| \nabla \Theta_t \|_{L^2} \Big( \| \nabla p_{t} \|_{L^2} \|  \Delta p_{t} \|_{L^2} + \|\Delta p  \|_{L^2} \| \nabla p_{tt} \|_{L^2}\Big)\| \Delta p_{tt} \|_{L^2}.
\end{aligned}
\end{equation}
Applying Young's  inequality then  yields 
\begin{equation} \label{R31_ineq}
\begin{aligned}
\vert R_{31} \vert \lesssim &\,  \| \Delta p \|_{L^2}^2 \| p_{ttt} \|_{L^2}^2 + \|\nabla p_t \|_{L^2}^2 \| \nabla p_{tt} \|_{L^2}^2 \\
&\,+ ( 1+ \| \Delta  \Theta \|_{L^2}^{1+ \gamma_2})^2 \| \nabla \Theta_t \|_{L^2}^2 \Big( \| \nabla p_{t} \|_{L^2}^2 \|  \Delta p_{t} \|_{L^2}^2 + \|\Delta p  \|_{L^2}^2 \| \nabla p_{tt} \|_{L^2}^2 \Big) \\
&\, +\varepsilon\| \Delta p_{tt} \|_{L^2}^2\\
\lesssim &\, \big(1+  \mathcal{E}[\Theta]+ \mathcal{E}[\Theta]^{2+\gamma_2} \big)E_1[p] \big(D_1[p]+D_2[p] \big) +\varepsilon\| \Delta p_{tt} \|_{L^2}^2.
\end{aligned}
\end{equation}
As for the integral $R_{32}$, we have
\begin{equation}
\begin{aligned}
\vert R_{32} \vert \leq &\, \Big( \| \tilde{h}' (\Theta) \|_{L^\infty}  \| \Theta_t \|_{L^4}  \|  \Delta p \|_{L^4}  + \| \tilde{h} (\Theta) \|_{L^\infty}  \| \Delta p_t  \|_{L^2} \Big) \| \Delta p_{tt} \|_{L^2}.
\end{aligned}
\end{equation}
Obviously estimating $R_{32}$ does not differ much from estimating $R_{22}$, which means that we can follow the same strategy to reach the bound  
\begin{equation} \label{R32_ineq}
\begin{aligned}
\vert R_{32} \vert \lesssim &\,\big(1+ \mathcal{E}[\Theta]^{1+ \gamma_1} \big) \big( E_1[p] + (E_1[p])^{1-\frac{d}{4}}(E_2[p])^{\frac{d}{4}} + (\mathcal{E}_0[\Theta])^{1-\frac{d}{4}} (\mathcal{E}[\Theta])^{\frac{d}{4}}  \big)\\
&\times \big( D_1[p]+ \mathcal{D}[\Theta] \big)+ \varepsilon \| \Delta p_{tt} \|_{L^2}^2.
\end{aligned}
\end{equation}
Plugging \eqref{R31_ineq}, \eqref{R32_ineq} into \eqref{identity_delta_p_tt} and choosing $\varepsilon$ small enough, we conclude
 \begin{equation} \label{estimate_delta_p_tt}
\begin{aligned}
\frac{1}{2} &\, \ddt \Big( \| \nabla p_{tt} \|_{L^2}^2 +h(0) \| \Delta p_t \|_{L^2}^2 \Big) +\frac{b}{2} \| \Delta p_{tt} \|_{L^2}^2\\
 \lesssim &\, \big(1+  \mathcal{E}[\Theta]+ \mathcal{E}[\Theta]^{2+\gamma_2} + \mathcal{E}[\Theta]^{1+ \gamma_1} \big)\\
 & \times \big( E_1[p] + (E_1[p])^{1-\frac{d}{4}}(E_2[p])^{\frac{d}{4}}+ (\mathcal{E}_0[\Theta])^{1-\frac{d}{4}} (\mathcal{E}[\Theta])^{\frac{d}{4}} \big) \big(D_1[p]+D_2[p] + \mathcal{D}[\Theta]  \big)  .
\end{aligned}
\end{equation}
Lastly, collecting the estimates \eqref{estimate_nabla_delta_p}, \eqref{estimate_p_ttt} and \eqref{estimate_delta_p_tt}, we obtain
\begin{equation}
\begin{aligned}
 &\ddt E_2[p](t) + D_2[p]\\
\lesssim&\, D_1[p] + \,\big(1+  \mathcal{E}[\Theta]+ \mathcal{E}[\Theta]^{2+\gamma_2} + \mathcal{E}[\Theta]^{1+ \gamma_1} \big)\\
&\, \times \big( E_1[p] + (E_1[p])^{1-\frac{d}{4}}(E_2[p])^{\frac{d}{4}}+ (\mathcal{E}_0[\Theta])^{1-\frac{d}{4}} (\mathcal{E}[\Theta])^{\frac{d}{4}} \big) \big(D_1[p]+D_2[p] + \mathcal{D}[\Theta]  \big)  , 
\end{aligned}
\end{equation}
which after integration in time results in the bound \eqref{third_estimate_p}. This  ends the proof of Lemma \ref{lemma5}.
 \end{proof}
 \section{Global well-posedness and exponential decay}\label{Sec_Global Existence}
 
  The main goal in this section is to prove Theorems \ref{Them_Global_solutions} and \ref{Them_exponential_decay}. 
 To show global  well-posedness for small initial data with smallness imposed only  on a lower-order norm, see e.g., \cite{LasieckaOng, bongarti2021vanishing}, we define the combined lower-order energy, which is the energy norm that we assume to be small: 
 \begin{subequations} 
\begin{equation} \label{low_total_energy_p_theta}
\mathsf{E_{low}}[p, \Theta](t):= E_{1}[p](t) + \mathcal{E}_0[\Theta](t),
\end{equation}
and the higher energy norm, which indicates the regularity needed to prove the main result   
\begin{equation} \label{high_total_energy_p_theta}
\mathsf{E_{high}}[p, \Theta](t):= E_{1}[p](t)+ E_{2}[p](t)+ \mathcal{E}[\Theta](t).
\end{equation}   
The associated dissipation rates are given respectively as 
\begin{equation} \label{total_dissip_p_theta}
\begin{aligned}
\mathsf{D_{low}}[p, \Theta](t)&\,:= D_{1}[p](t) + \mathcal{D}_0[\Theta](t),\vspace{0.2cm}\\
\mathsf{D_{high}}[p, \Theta](t)&\,:= D_{1}[p](t)+ D_{2}[p](t)+ \mathcal{D}[\Theta](t).
\end{aligned}
\end{equation}
\end{subequations}
\subsection{Proof of Theorem \ref{Them_Global_solutions}}\label{Section_Proof_Global}
 The proof follows from the local existence, uniform-in-time a priori estimates, and a  continuity argument. We define 
\begin{equation}
\begin{aligned}
\mathbf{E}_{\mathsf{low}}(t)=&\,\sup_{0\leq s\leq t}\mathsf{E_{low}}[p, \Theta](s),\qquad \mathbf{E}_{\mathsf{high}}(t)=\,\sup_{0\leq s\leq t}\mathsf{E_{high}}[p, \Theta](s) , \\
\mathbf{D}_{\mathsf{low}}(t)=&\, \int_0^t  \mathsf{D_{low}}[p, \Theta](s) \ds,\qquad \mathbf{D}_{\mathsf{high}}(t)=\, \int_0^t  \mathsf{D_{high}}[p, \Theta](s) \ds, \qquad t \geq 0. 
\end{aligned}
\end{equation}
According to what we have proven in the previous section, we can assemble the estimates on $E_{1}[p]$ and $\mathcal{E}_0[\Theta]$ to yield a uniform  bound for $\mathsf{E_{low}}[p, \Theta]$ for all $0 \leq s \leq t$ 
\begin{equation} \label{estimate_combined_E_low}
\begin{aligned}
 \mathsf{E_{low}}[p, \Theta](s) + \int_0^s &\, \mathsf{D_{low}}[p, \Theta](r) \,  \textup{d} r\\
\leq C_1 \Big\{ \mathsf{E_{low}}[p, \Theta](0) + \int_0^s &\,  \Big( 1 +\mathsf{E_{high}}[p, \Theta] + \mathsf{E_{high}}[p, \Theta]^{2+ \gamma_2}+\mathsf{E_{high}}[p, \Theta]^{ 1+ \gamma_1 } \Big)\\
&\times \Big( \mathsf{E_{low}}[p, \Theta] + (\mathsf{E_{low}}[p, \Theta])^{1-\frac{d}{4}} (\mathsf{E_{high}}[p, \Theta])^{\frac{d}{4}}\Big) \mathsf{D_{low}}[p, \Theta](r) \,  \textup{d} r \Big\}, 
\end{aligned}   
\end{equation}
which in turn implies
\begin{equation}\label{low_Estimate}
\begin{aligned}
\mathbf{E}_{\mathsf{low}}(t)+\mathbf{D}_{\mathsf{low}}(t)\leq&\, C_1\mathsf{E_{low}}[p, \Theta](0)\\    
&+ C_1\Big( 1 +\mathbf{E}_{\mathsf{high}}(t) + (\mathbf{E}_{\mathsf{high}}(t))^{2+ \gamma_2}+(\mathbf{E}_{\mathsf{high}}(t))^{ 1+ \gamma_1 } \Big)\\
&\times\, \Big( \mathbf{E}_{\mathsf{low}}(t) + (\mathbf{E}_{\mathsf{low}}(t))^{1-\frac{d}{4}} (\mathbf{E}_{\mathsf{high}}(t))^{\frac{d}{4}}\Big) \mathbf{D}_{\mathsf{low}}(t). 
\end{aligned}
\end{equation} 
In addition, by adding \eqref{main_estimate_energy_theta} and  \eqref{main_est_energy_p}, we get the following estimate for the higher-order energy $\mathsf{E_{high}}[p, \Theta]$ for $0 \leq s \leq t$  
\begin{equation} \label{estimate_combined_E_high}
\begin{aligned}
 \mathsf{E_{high}}[p, \Theta](s) + \int_0^s &\, \mathsf{D_{high}}[p, \Theta](r)\,  \textup{d} r\\
\leq  C_2 \Big\{ \mathsf{E_{high}}[p, \Theta](0)+ &\, \int_0^s  \Big( 1 +\mathsf{E_{high}}[p, \Theta] + \mathsf{E_{high}}[p, \Theta]^{2+ \gamma_2}+\mathsf{E_{high}}[p, \Theta]^{ 1+ \gamma_1 } \Big)\\
&\times \Big( \mathsf{E_{low}}[p, \Theta] + (\mathsf{E_{low}}[p, \Theta])^{1-\frac{d}{4}} (\mathsf{E_{high}}[p, \Theta])^{\frac{d}{4}}\Big) \mathsf{D_{high}}[p, \Theta](r)\,  \textup{d} r \Big\}.
\end{aligned}
\end{equation}
This also means the following bound for $\mathbf{E}_{\mathsf{high}}(t)$ 
\begin{equation}\label{E_high}
\begin{aligned}
\mathbf{E}_{\mathsf{high}}(t)+\mathbf{D}_{\mathsf{high}}(t)\leq&\, C_2\mathsf{E_{high}}[p, \Theta](0)\\
&+C_2\Big( 1 +\mathbf{E}_{\mathsf{high}}(t) + (\mathbf{E}_{\mathsf{high}}(t))^{2+ \gamma_2}+(\mathbf{E}_{\mathsf{high}}(t))^{ 1+ \gamma_1 } \Big)\\
&\times\, \Big( \mathbf{E}_{\mathsf{low}}(t) + (\mathbf{E}_{\mathsf{low}}(t))^{1-\frac{d}{4}} (\mathbf{E}_{\mathsf{high}}(t))^{\frac{d}{4}}\Big) \mathbf{D}_{\mathsf{high}}(t). 
\end{aligned}
\end{equation}
Let us make the following a priori assumption:
\begin{equation}\label{Assumption_Small}
\mathbf{E}_{\mathsf{low}}(t)\leq \eta,\qquad  \mathbf{E}_{\mathsf{high}}(t)\leq M,\qquad 2 \|k(\Theta) p\|_{L^\infty L^\infty}\leq \mathsf{m} <1
\end{equation}
for constants $\eta,\,  M, \mathsf{m}>0$ with $\eta$ and $\mathsf{m}$ small enough. Then, estimate \eqref{low_Estimate} implies 
\begin{equation}\label{low_Estimate_2}
\begin{aligned}
\mathbf{E}_{\mathsf{low}}(t)+\mathbf{D}_{\mathsf{low}}(t)\leq&\, C_1\mathsf{E_{low}}[p, \Theta](0)\\    
&+ C_1( 1 +M + M^{2+ \gamma_2}+M^{ 1+ \gamma_1 } )
 ( \eta + \eta^{1-\frac{d}{4}} M^{\frac{d}{4}}) \mathbf{D}_{\mathsf{low}}(t). 
\end{aligned}
\end{equation}
Similarly, from \eqref{E_high}, we have 
\begin{equation}\label{E_high_2}
\begin{aligned}
\mathbf{E}_{\mathsf{high}}(t)+\mathbf{D}_{\mathsf{high}}(t)\leq&\, C_2\mathsf{E_{high}}[p, \Theta](0)\\
&+C_2( 1 +M + M^{2+ \gamma_2}+M^{ 1+ \gamma_1 } )
 ( \eta + \eta^{1-\frac{d}{4}} M^{\frac{d}{4}}) \mathbf{D}_{\mathsf{high}}(t), 
\end{aligned}
\end{equation}
which gives if $\eta$ and $\mathsf{m}$ are small enough 
\begin{equation}
\mathbf{E}_{\mathsf{low}}(t)+\mathbf{D}_{\mathsf{low}}(t)\leq \, \tilde{C}_1\mathsf{E_{low}}[p, \Theta](0), 
\end{equation}
and 
\begin{equation}
\mathbf{E}_{\mathsf{high}}(t)+\mathbf{D}_{\mathsf{high}}(t)\leq \tilde{C}_2\mathsf{E_{high}}[p, \Theta](0).
\end{equation}
Furthermore, we have by the Sobolev embedding  
\begin{equation}
\|k(\Theta) p\|_{L^\infty L^\infty}\leq k_1\|p\|_{L^\infty L^\infty}\lesssim \mathbf{E}_{\mathsf{low}}(t)\lesssim \mathsf{E_{low}}[p, \Theta](0). 
\end{equation}
Therefore, in the standard way, see e.g. \cite{Tao2006NonlinearDE} and  \cite[Section 4]{LasieckaOng}, as long as $\mathsf{E_{low}}[p, \Theta](0)$ is sufficiently small and $\mathsf{E_{high}}[p, \Theta](0)$ is bounded, say; by $M_0$ with $M_0 \leq \frac{M}{2\tilde{C}_2}$, then the above uniform a priori estimates obtained under the assumption \eqref{Assumption_Small} imply the global existence of the solution. This completes the proof of Theorem \ref{Them_Global_solutions}.

\subsection{Proof of Theorem \ref{Them_exponential_decay}} Let $(p, \Theta)$ be the global solution whose existence is provided by Theorem \ref{Them_Global_solutions}. Going back to \eqref{estimate_combined_E_high} and integrating from $s$ to $t$ in all the estimates leading to \eqref{estimate_combined_E_high}, we have, by taking account of \eqref{Assumption_Small}, the bound
 \begin{equation} \label{estimate_combined_E_high_1}
\begin{aligned}
 \mathsf{E_{high}}[p, \Theta](t) + \int_s^t &\, \mathsf{D_{high}}[p, \Theta](r)\,  \textup{d} r\\
\leq  C_2  \mathsf{E_{high}}[p, \Theta](s)+ &\,  C_2( 1 +M + M^{2+ \gamma_2}+M^{ 1+ \gamma_1 } )( \eta + \eta^{1-\frac{d}{4}} M^{\frac{d}{4}})\int_s^t \mathsf{D_{high}}[p, \Theta](r)\,  \textup{d} r.
\end{aligned}
\end{equation}
We can take $\eta$ sufficiently small so that 
\begin{equation}
2C_2( 1 +M + M^{2+ \gamma_2}+M^{ 1+ \gamma_1 } )( \eta + \eta^{1-\frac{d}{4}} M^{\frac{d}{4}}) \leq 1
\end{equation} 
is satisfied. Thus, the estimate \eqref{estimate_combined_E_high_1} becomes 
\begin{equation} \label{estimate_combined_E_high_2}
\begin{aligned}
 \mathsf{E_{high}}[p, \Theta](t) + \int_s^t &\, \mathsf{D_{high}}[p, \Theta](r)\,  \textup{d} r \lesssim  \mathsf{E_{high}}[p, \Theta](s).
\end{aligned}
\end{equation}   
On the other hand, it is clear from the definition of $\mathcal{E}[\Theta]$ and $\mathcal{D}[\Theta]$ that (cf. \eqref{heat_energy_1}, \eqref{heat_dissip_1} and \eqref{heat_energy_total})
\begin{equation} \label{energy_dissip_theta}
\mathcal{E}[\Theta](t) \lesssim \mathcal{D}[\Theta](t), \qquad t \geq 0.
\end{equation}
Moreover, we have by using Poincar\'{e}'s inequality  and elliptic regularity  
\begin{equation}
\| p_{tt} \|^2_{L^2} \lesssim \| \nabla p_{tt} \|^2_{L^2}, \qquad \quad \| \nabla  p_t \|^2_{L^2} \lesssim \| \Delta p_t \|^2_{L^2},
\end{equation}
then these inequalities together with the last bound in \eqref{Assumption_Small} yield 
\begin{equation} \label{energy_dissip_p}
 E_{1}[p](t)+ E_{2}[p](t) \lesssim D_{1}[p](t)+ D_{2}[p](t), \qquad t \geq 0.
\end{equation} 
Combining \eqref{energy_dissip_theta} and \eqref{energy_dissip_p} and recalling  \eqref{high_total_energy_p_theta} and \eqref{total_dissip_p_theta}, we obtain 
\begin{equation}
\mathsf{E_{high}}[p, \Theta](t) \lesssim \mathsf{D_{high}}[p, \Theta](t).
\end{equation}
Inserting the above inequality in \eqref{estimate_combined_E_high_2}, it follows 
\begin{equation} \label{}
\begin{aligned}
 & \mathsf{E_{high}}[p, \Theta](t) + \int_s^t \mathsf{E_{high}}[p, \Theta](s) \,  \textup{d} r \lesssim \mathsf{E_{high}}[p, \Theta](s), \quad 0 \leq s <t,
\end{aligned}
\end{equation}
thereby we deduce that the energy decays exponentially fast by applying Lemma \ref{Gronwall}.

\end{document}